\documentclass[reqno,10pt]{amsart}
\usepackage[sort]{cite}
\usepackage{amsmath, amsthm, amsfonts}
\usepackage{amssymb}
\usepackage{pdfsync}
\usepackage{enumitem}
\usepackage{parskip}
\usepackage{hyperref}
\usepackage{mathrsfs}  
\usepackage{mathtools}
\usepackage{mdframed}
\usepackage{tikz}
\usepackage{ctable}
\usepackage{caption}
\usepackage{subcaption}
\usepackage[margin=1in]{geometry}
\mathtoolsset{showonlyrefs}
\numberwithin{equation}{section}
\newtheorem{theorem}{Theorem}[section]
\newtheorem{prop}[theorem]{Proposition}
\newtheorem{definition}[theorem]{Definition}
\newtheorem{lem}[theorem]{Lemma}
\newtheorem{cor}[theorem]{Corollary}

\theoremstyle{remark}
\newtheorem{remark}[theorem]{Remark}

\allowdisplaybreaks

\newcommand{\cA}{\mathcal A}

\newcommand{\cH}{\mathcal H}

\newcommand{\cL}{\mathcal L}
\newcommand{\cM}{\mathcal M}
\newcommand{\cN}{\mathcal N}

\newcommand{\cS}{\mathcal S}

\newcommand{\cX}{\mathcal X}
\newcommand{\cY}{\mathcal Y}

\def\M{\mathsf{M}}

\def\p{\mathsf{p}}

\def\R{\mathbb{R}}

\def\N{\mathbb{N}}

\def\cA{\mathcal{A}}
\def\sA{\mathscr{A}}

\def\a{\mathfrak{a}}

\def\div{{\rm div}}

\def\min{{\rm min}}
\def\max{{\rm max}}
\def\dom{{\rm dom}}

\def\Per{{\rm  Per}}

\def\Inp{I_{\rm np}}
\def\Ip{I_{\rm p}}

\def\Enp{E_{\rm np}}

\def\cYk{\cY_k}
\def\Tr{{\rm Tr}}
\def\um{u_{{\rm min}}}
\def\umk{u_{{\rm min},k}}
\def\pmi{\psi_{{\rm min}}}
\def\pmk{\psi_{{\rm min},k}}
\def\ns{\nu_{\Sigma_1}}

\def\Hz{H_{x_0}}

\begin{document}

\title[A new approach to solvation models]{A new approach to constrained total variation solvation models and  the study of solute-solvent interface profiles}

%\author{
%	Zhan Chen $^{1}$ and
%	%\address{
%	Yuanzhen Shao$^2$,\footnote{Please address correspondence to Yuanzhen Shao. E-mail:yshao8@ua.edu}\\
%	$^1$ 	Department of Mathematical Sciences \\
%	Georgia Southern University, Statesboro, GA \\
%	$^2$ 	Department of Mathematics, University of Alabama, Tuscaloosa, AL\\ 
%}

\author{Zhan Chen}
\address{Department of Mathematical Sciences\\
	Georgia Southern University\\
	Statesboro, Georgia\\
        USA}
\email{zchen@georgiasouthern.edu}

\author{Yuanzhen Shao$^{\star}$}
\address{Department of Mathematics\\
The University of Alabama\\ 
	Tuscaloosa, Alabama \\
	USA}
\email{yshao8@ua.edu}

\thanks{$^{\star}$Corresponding author.}
%%edit model and parameters and diagram
%% 

%\date{\today}
%\maketitle

\begin{abstract}
%A new constrained variational implicit solvation model (VISM) is introduced to address several open questions concerning the existing diffuse-interface based VISMs	.
In the past   decade, variational implicit solvation models (VISM) have achieved great success in solvation energy predictions.
However, all existing   VISMs in literature lack the uniqueness of an energy minimizing solute-solvent interface and thus prevent us from studying many important properties of   the  interface profile.
To overcome this difficulty, we introduce a new constrained VISM and conduct a rigorous analysis of the model.
Existence, uniqueness and regularity of the energy minimizing   interface has been studied. A necessary condition for the formation of a sharp solute-solvent interface has been derived.
Moreover, we develop a novel   approach to the variational analysis of the constrained model, which provides a complete answer to a question  in  our previous work~\cite{Shao2022constrained}. 
Model validation and numerical implementation have been demonstrated by using several common biomolecular modeling tasks.
Numerical simulations show that the solvation energies calculated from our new model match the experimental data very well.
\end{abstract}

 %35Q75;49S05;49Q20;

\maketitle

%\tableofcontents

%%%%%%%%%%%%%%%%%%%%%%%%%%%%%%%%%%%%%%%%%%%%%%%%%%%%%%%%%%%%%%%%
\section{Introduction}\label{Section 1}

The description of the complex  interactions between the solute and solvent plays an important role in essentially all chemical and biological processes.
Solute-solvent interactions are typically described by solvation energies (or closely related quantities):
the free energy of transferring the solute (e.g.  macromolecules including proteins, DNA, RNA) from the vacuum to a solvent environment of interest (e.g.  water at a certain ionic strength). 
There are two major approaches for solvation energy analysis, i.e., explicit solvent models and implicit solvent models \cite{doi:10.1021/ct400065j}.
Explicit models, treating solvent as individual molecules, are too computationally expensive for large  solute-solvent systems, such as the solvation of macromolecules in ionic environments; in contrast, implicit models, by averaging the effect of solvent phase as continuum media \cite{Feig:2004b,Baker:2005,Boschitsch:2004, Baker10037, BOTELLOSMITH2013274, doi:10.1021/jp7101012, https://doi.org/10.1002/jcc.23033}, are much more efficient and thus are able to handle much larger systems \cite{Baker10037, doi:10.1063/1.4745084,Grochowski:2007,Lamm:2003,Fogolari:2002,Tjong:2007b,Mongan:2007,Grant:2007}.

Central in the description of the solvation energy in implicit solvent models is an interface separating the discrete solute and the continuum solvent domains.
All of the physical properties of interest, including electrostatic free energies, biomolecular surface areas, molecular cavitation volumes    and p$K_a$ values are very sensitive to the interface definition \cite{Dong:2006,Swanson:2005a,Wagoner:2006}. 
Variational implicit solvation models (VISM) stand out as a successful approach to compute the disposition of an interface separating the  solute and the solvent \cite{Wei:2005,Bates:2008,doi:10.1063/1.2171192,doi:10.1063/1.2171192,Cheng:2007e,CChen:2009, ZhanChen:2010a, Zhou2020curvature,MR3740372}.
In a VISM, the desired interface profile is obtained by  minimizing a solvation energy functional coupling the discrete description of solute and the continuum description of solvent.
% in addition to polar and nonpolar interactions.

Despite of their initial successes in solvation energy calculations,  sharp solute-solvent interface  models suffer from several drawbacks.
Firstly, from a physical point of view, there should be a smooth transition region, in which atoms of solute and solvent are mixed. 
In principle, an isolated molecule can be analyzed by the first principle --- a quantum mechanical description of the wave function or density distribution of all the electrons and nuclei. However, such a description is computationally intractable for large biomolecules. Under physiological conditions, biomolecules are in a non-isolated environment, and are interacting with solvent molecules and/or other biomolecules. Therefore, their wave functions overlap spatially, so do their electron density distributions.
Secondly, from an analytic point of view, the presence of geometric singularities is inevitable in many conventional VISMs. It makes the underlying model  lack   stability and differentiability, which generates an intrinsic difficulty  in the rigorous analysis of the model. 
Thirdly, from a computational point of view, these surface configurations produce  fundamental difficulty in the simulation of the governing partial differential equations (PDEs), like the Poisson-Boltzmann (PB) equation.
Those considerations motivate  the use of the  diffuse solvent-solute interface definition.
%This calls for the need of  a new notion of solvent-solute interface, which is known as the diffuse interface. 
%; and as a direct consequence, various regularization techniques need to be implemented to approximate the sharp interface  with diffuse interface profiles, which further increases the inaccuracy of the simulation results. 
%Those considerations motivate  the use of alternative diffuse solvent-solute interface definitions.

Among all effort  to ameliorate the solvent-solute interface definition, arguably, one of the most extensively used models is the total variation based model (TVBVISM), cf. \cite{ZhanChen:2010a,PhysRevLett.96.087802, Wei2010, Wei_2016_differential,Wang_2015_Parameter,Zhao_2011_Pseudo}. The main idea of TVBVISM is based on a transition parameter $u: \Omega \to [0,1]$ such that $u $ takes value $1$ in the solute and  $0$ in the solvent  region. 
%The surface area  is replaced by
%$ 
%\int_\Omega |\nabla u|\, dx,
%$ 
%which, by the coarea formula,  
%approximates the mean surface area of the level sets of $u$, cf. \cite{Wei2010}. 
%Other terms in TVBVISM are defined by multiplying each energy density    by $u$ if  defined in the solute or  by $ (1-u) $ if in the solvent; and the electrostatic potential $\psi$ solves a corresponding generalized PB equation. {\color{red} The following formulation needs to be rewritten from previous article}. 
More precisely, the following total solvation free energy was proposed in terms of $u$: 
\begin{align}\label{totalfunctional1}
	\notag I
	=&   \gamma \|Du\|(\Omega) + \int_\Omega P_h u(x)dx+\int_\Omega \rho_s (1-u(x))U^{\mathrm{vdW}}(x)\, dx \\
	\notag  &  +\int_{\Omega}\left\{u(x)\left[\rho_m(x)\psi(x)-\frac{1}{2}\epsilon_m|\nabla
	\psi (x)|^{2}\right]\right.  \\
	&   \left.+(1-u(x))\left[-\frac{1}{2}\epsilon_s|\nabla \psi (x) |^{2}- \beta^{-1} \sum\limits_{j=1}^{N_c} c_j^\infty  (e^{- \beta q_j \psi(x	)    }-1) \right]\right\} dx.
\end{align}
Here the constant $\gamma>0$ is the  surface tension. By the coarea formula for a Lipschitz function  $u: \Omega\to [0,1]$,
$$
\|Du\|(\Omega):=\int_\Omega d|D u|=\int_0^1 \cH^2(\Omega \cap u^{-1}(t))\,dt,
$$
where $\cH^2$ stands for the 2-dimensional Hausdorff measure. 
Hence, the total variation term $\|Du\|(\Omega)$ represents the mean surface area of a family of isosurfaces $\Omega \cap u^{-1}(t)$.
See \cite{Wei2010} for more detail.
According to this geometric interpretation, 
$ 
\gamma \|Du\|(\Omega),
$  
measures the disruption of intermolecular and/or intramolecular bonds during the solvation process.

The constant $P_h$ is the hydrodynamic pressure.
In a previous work \cite{Shao2022constrained}, we proposed a novel physical interpretation of the characteristic function $u$ so that $u(x)$ represents the volume ratio of the solute at   $x\in \Omega$.
Therefore, 
$   \int_{\Omega} P_h  u \, dx  $ is the mechanical work of creating the biomolecular size vacuum
in the solvent.
$\rho_s$ is the constant solvent bulk density,  and $U^{\mathrm{vdW}} (x)$ is the attractive portion of the Van der Waals potential at point $x$.
It represents the attractive dispersion effects near the solute-solvent interface and
has been shown by Wagoner and Baker\cite{Wagoner:2006} to play a crucial role in  accurate nonpolar solvation analysis. 
The first three terms are usually termed the nonpolar portion of the solvation free energy.

The second and third lines of \eqref{totalfunctional1} are usually called the polar portion of the solvation free energy, in which
$\psi$ is the electrostatic potential. $\rho_m$ is an $L^\infty$-approximation of  the density of molecular charges;
$\epsilon_m$ and $\epsilon_s$ are the dielectric constants of the solute molecule and the solvent, respectively, with $0<\epsilon_m\ll \epsilon_s$. 
%Usually, $\epsilon_m\approx 1$ for the protein and $\epsilon_s\approx 80$ for the water.
$q_j$ is the
charge of ion species $j=1,2,\cdots,N_c$; and
$c_j^\infty $ is the bulk concentration of the $j$-th ionic species. 
Finally, $\beta=1/k_B T$,  where $k_B$ is the Boltzmann constant and $T$ is the  absolute  temperature.
For notational brevity, throughout this paper, we put 
\begin{equation}
	\label{Def B}
	B(s)= \beta^{-1} \left[ \sum\limits_{j=1}^{N_c} c_j^\infty  \left( e^{- \beta s q_j } -1 \right) \right].
\end{equation}
Numerical simulations show  that diffuse-interface models can significantly improve the accuracy and efficiency of solvation energy computation \cite{Wei:2005,Bates:2008,doi:10.1063/1.2171192,doi:10.1063/1.2171192,Cheng:2007e,CChen:2009, ZhanChen:2010a, Zhou2020curvature,MR3740372, MR3396402}.
In contrast, on a theoretic level,   there are several open questions concerning  model \eqref{totalfunctional1}.

First, the uniqueness of a minimizer is unknown for \eqref{totalfunctional1}.
Indeed, most of the solvation energy functionals, regardless of sharp or diffuse interfaces, only predict local minimizers, cf. \cite{Wei:2005,Bates:2008,doi:10.1063/1.2171192,Cheng:2007e,CChen:2009, ZhanChen:2010a, Zhou2020curvature, MR3740372, MR3396402}. As a consequence,     solutions of the corresponding Euler-Lagrange equations may not correctly depict   the energy minimizing interface profile. 
In contrast, any minimizer of  \eqref{totalfunctional1} is global.
However, lacking strict convexity, \eqref{totalfunctional1} may admit multiple global  minimizers.
This prevents us from studying many properties of the  interface profile, e.g. the size of the set of  discontinuities.
%As a consequence, the corresponding depiction of the diffuse solute-solvent interface, is still absent.
These observations motivate  us to introduce  strict convexity into model \eqref{totalfunctional1} by including a new parameter $p=\frac{2N}{2N-1}$ with $N\in \N$ so that $u^p(x)$ represents the volume ratio of the solute at $x\in \Omega$. It is important to notice that the  geometric meaning of the term $\|D u\|(\Omega)$ remains the same as in the original model \eqref{totalfunctional1}.
We will establish the existence, uniqueness and regularity of the global minimizer of the modified model,  see \eqref{non penalized total energy functional final form}.

Second, the natural admissible space to minimize \eqref{totalfunctional1} is the space of    $BV-$functions.
Therefore, it is possible that   model~\eqref{totalfunctional1} is minimized by the characteristic function of a set of finite perimete.
This corresponds to a sharp solute-solvent interface, an unrealistic  situation as discussed before.
Nevertheless, it is mathematically impossible to exclude such situations in model~\eqref{totalfunctional1} due to the lack of uniqueness of a minimizer.
Based on the modified model, this work provides  a partial answer to the question why the   solvation free energy is not minimized by a sharp   interface. 
More precisely,   we  show  that a necessary condition for a nonpolar molecule to have a sharp energy-minimizing interface is  that the mean curvature of its  Van Der Waals surface  is everywhere nonpositive, which is unrealistic for almost all real-world biomolecules.
To the best of our knowledge, our work is the first to give a mathematical explanation of such phenomenon.

Third, the physical meaning of the characteristic function $u$ enforces two biological constraints: (1)  $u$ needs to be 1 for the pure solute region and 0 in the pure solvent area, and (2) as a volume ratio function, it must   satisfy that $0\leq u\leq 1. $
This leads to a constrained total variation model~\eqref{non penalized total energy functional final form}, which is a non-differentiable functional  with  a two-sided  obstacle.
It is known that the Euler-Lagrange equations  of similar functionals with simpler structure and without obstacle, e.g. Rudin-Osher-Fatemi models, were formally derived by using the $1-$Laplacian operator \cite{MR3363401}.
With the presence of the obstacle, on a heuristic level with sufficiently smooth minimizer $u$  and energy  functional,
one expects the corresponding first variations  with respect to $u$ to take the form of a variational inequality, or equivalently,
of a $1-$Laplacian type equation involving a measure supported on the coincidence sets $\{u=0\}$ and $\{u=1\}$.
Unfortunately,  both the functional ~\eqref{non penalized total energy functional final form} and the minimizer $u$  lack the  required smoothness. 
This casts a shadow over the study of the first variations of the constrained total variation model, not even formally.
%However,   including   such constraints    generates an essential difficulty for the variational analysis of the new model \eqref{non penalized total energy functional final form}. 
%Indeed, to the best of our knowledge, the variational 
%even the polar function of such a complicated constrained total variation problem  has never been investigated in literature.
%This makes the numerical computations of the solvation energy a challenging task.
In \cite{Shao2022constrained}, we proposed a novel approach to the variational analysis of such constrained VISM via approximation by a sequence of $q$-energy type functionals. This approach was applied to the numerical study of the nonpolar energy in our previous work \cite{Shao2022constrained}. 
Using a similar idea and   the new volume ratio function $u^p$, we will   rigorously derive the variational formulas of the new total energy functional.
%{\color{red}
%Comments on numerical simulations can  be put here.
%}

%%%%%%%%%%%%%%%%%%%%%%%%%%%%%%%%%%%%%%%%%
%Organization of the paper
%%%%%%%%%%%%%%%%%%%%%%%%%%%%%%%%%%%%%%%%%
The rest of the paper is organized as follows.  A list of the main theorems  is stated at the end of the introduction.
In Section~\ref{Section 2}, we state   the precise definition of our new model.
In Section~\ref{Section 3}, we study a family of perturbed Poisson-Boltzmann equations. These equations will be used in Sections~\ref{Section 4} and \ref{Section 6}.
Section~\ref{Section 4} is devoted to the validation of the model, in which we prove the existence and uniqueness of a minimizer and the continuous dependence of the   solvation energy on the biological constraints.
In Section~\ref{Section 5},  a necessary condition for the formation of a sharp solute-solvent interface is derived.
The argument heavily relies on the tools from nonsmooth convex analysis.
In Section~\ref{Section 6}, we conduct a variational analysis of our new   model by means of an approximation argument.
Base on this  analysis, our model, including its solvation energy and solute-solvent interface predictions, is studied through numerical simulations. 
For the readers' convenience, we include two appendices at the end of this article, one on $BV-$functions and the other on nonsmooth convex analysis.

%%%%%%%%%%%%%%%%%%%%%%%%%%%%%%%%%%%%%%%%%
%List of main theorems
%%%%%%%%%%%%%%%%%%%%%%%%%%%%%%%%%%%%%%%%%
For the reader's convenience, we will give a list of the main theoretic results here:
\begin{itemize}
\item Theorem~\ref{Thm: existence of global minimizer}: the existence and uniqueness of global minimizers of the total solvation energy;
%\item Theorem~\ref{Thm: uniqueness} uniqueness of global minimizer of the total solvation energy;
\item Theorem~\ref{Thm: free energy converg}: the continuous dependence of the solvation energy  on the biological constraints;
\item Theorem~\ref{Thm: sharp interface}: a  necessary condition for the formation of a sharp solute-solvent interface;
\item Theorem~\ref{Thm: existence of minimizer Ek}: the theoretic basis of the numerical simulations.
\end{itemize}

%%%%%%%%%%%%%%%%%%%%%%%%%%%%%%%%%%%%%%%%%%%%%%%%%%%%%%%%%%%%%%%%
\section{Solvation Free Energy Functional}\label{Section 2}

%%%%%%%%%%%%%%%%%%%%%%%%%%%%%%%%%%%%%%%%%
%Notations
%%%%%%%%%%%%%%%%%%%%%%%%%%%%%%%%%%%%%%%%%
\subsection{Notations}\label{Section 2.0}
In this article,  we use $x=(x_1, x_2,\cdots, x_N)$ to denote the coordinates in $\R^N$. 
%$\partial_i$ stands for the partial derivative with respect to the $i-$th variable. 
$\mathbb{S}^{N-1}$ denotes the $(N-1)-$sphere in $\R^N$.
Given two vectors $u,v\in \R^N$, $u\cdot v$ is their inner products.
 
%Given any interval $I$ containing $0$, $\dot{I}:= I \setminus\{0\}$. 
Given  $U \subseteq \R^N$, 
%$\mathring{U}$ denotes the interior of $U$ and 
$\overline{U}$ stands for the closure of $U$. The topological boundary of $U$ is denoted by $\partial U$.
Given two domains $U$ and $\Omega $ in $\R^N$, $U\subset\subset \Omega$ means that  $\overline{U}\subset\Omega$.
%If $U$ consists of only one point, we set $\mathring{U}:=U$. 

%%%%%%%%%%%%%%%%%%%%%%%%%%%%%%%%%%%%%%%%%
For any two Banach spaces $X,Y$, the notation 
$$
X\hookrightarrow Y
%\qquad X\xhookrightarrow{d} Y
$$
means that $X$ is continuously embedded  in $Y$.
Given a sequence $\{u_k\}_{k=1}^\infty=(u_1,u_2,\cdots)$ in $X$, $u_k \rightharpoonup u$ in $X$ means that $u_k$ converge weakly to some $u\in X$.
%We denote by $X^*$ the topological dual of $X$ and $\langle \cdot , \cdot \rangle$ the duality pairing between $X$ and $X^*$. 

%%%%%%%%%%%%%%%%%%%%%%%%%%%%%%%%%%%%%%%%%
%$M(\Omega,X)$ denotes the set of all (Lebesgue) measurable functions defined on $\Omega$ taking value in $X$.

Given $1\leq p \leq \infty$, let $p' $ be its H\"older conjugate.
% and let $p^*$ be its Sobolev dual if $p< N $.
$L^p(U,X)$ is the set of all $X$-valued $p-$integrable (Lebesgue) measurable functions defined on $U$, whose norm is denoted by $\|\cdot\|_p$.
The notation $X$ is sometimes omitted when its choice is clear from the context.
$W^{k,p}(U)$ stands for  the Sobolev space consisting of functions whose weak derivatives up to $k-$th power belong to $L^p(U)$.
Additionally, $H^1(U)=W^{1, 2}(U)$.

%%%%%%%%%%%%%%%%%%%%%%%%%%%%%%%%%%%%%%%%%

Given two sets $A$ and $B$, $A\subseteq B$ and $A\subset B$ mean  that $A$ is a subset   and a proper subset of $B$, respectively.
 
Finally, we denote by $\cL^N$ and $\cH^{N-1}$ the $N-$dimensional Lebesgue measure and the $(N-1)-$dimensional Hausdorff measure, respectively.

%%%%%%%%%%%%%%%%%%%%%%%%%%%%%%%%%%%%%%%%%%%%%%%%%%%%%%%%%%%%%%%%
\subsection{An Experimental Based Domain Decomposition}\label{Section 2.1}

%%%%%%%%%%%%%%%%%%%%%%%%%%%%%%%%%%%%%%%%%%%%%%%
%Domain decomposition
%%%%%%%%%%%%%%%%%%%%%%%%%%%%%%%%%%%%%%%%%%%%%%%
Let $\Omega \subseteq \R^3$ be a bounded and connected Lipschitz domain  composed of three disjoint subdomains:
\begin{itemize}
\item $\Omega_m$:  solute (molecular) region;
\item $\Omega_s$: solvent  region;
\item $\Omega_t$: solute-solvent mixing  region.
\end{itemize}
We further assume that $\partial\Omega \subset \partial\Omega_s$ and $\partial\Omega_m \subset \partial\Omega_t$.
%We have $\overline{\Omega}=\overline{\Omega}_m \cup \overline{\Omega}_s\cup \overline{\Omega}_t$. 
Let 
$$
\Sigma_1=\partial\Omega_m  
$$ 
be  a smoothed Van Der Waals surface enclosing the pure solute region and 
$$
\Sigma_0=\partial\Omega_s \setminus \partial\Omega=\partial\Omega_t\setminus \Sigma_1
$$ 
be the smoothed solvent accessible surface outside which is the pure solvent domain. 
Suppose that $\Sigma_1 \cap \Sigma_0 =\emptyset$  and $\Omega_m, \Omega_s$ are non-empty. 
In addition, we assume that $\Sigma_i$, $i=0,1$, are embedded closed Lipschitz surfaces.
In this article, a closed surface always means one that is  compact, without boundary and embedded in $\R^3$.
We further assume that the solute region $\Omega_m$ contains $N_a$ solute atoms located at $x_1, \cdots ,  x_{N_a}$; and there are $N_c$   ion species outside $\Omega_m$. 
Finally, for notational brevity, we put $\Omega_w=\Omega\setminus \overline{\Omega}_s$. 
A picture illustration of the domain definition and decomposition can be found in Figure~\ref{domain}(A).

%\begin{figure}
%            \begin{center}
%                \begin{tabular}{cc}
%                         \includegraphics[width=0.5\columnwidth]{solute_solvent.png}
%                \end{tabular}
%            \end{center}
%             \caption{ Illustration of model domain definition and decomposition: $\Omega_m$: solute (molecular) region; $\Omega_s$: solvent  region; $\Omega_t$: solute-solvent mixing region.
%}\label{domain}
%\end{figure}

\begin{figure}[hbt!]
% \begin{center}
%                \begin{tabular}{cc}
%                         \includegraphics[width=0.4\columnwidth]{solute_solvent.png} & \includegraphics[width=0.5\columnwidth]{profile.png}
%                \end{tabular}
%            \end{center}
\begin{subfigure}{.5\textwidth}
  \centering
  \includegraphics[width=.7\columnwidth]{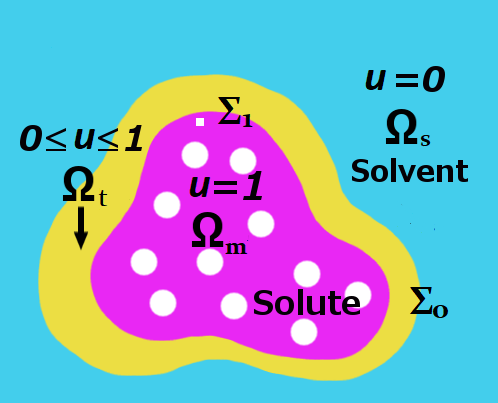}
   \caption{}
  %\label{fig:sub1}
\end{subfigure}%
\begin{subfigure}{.4\textwidth}
  \centering
  \includegraphics[width=1\columnwidth]{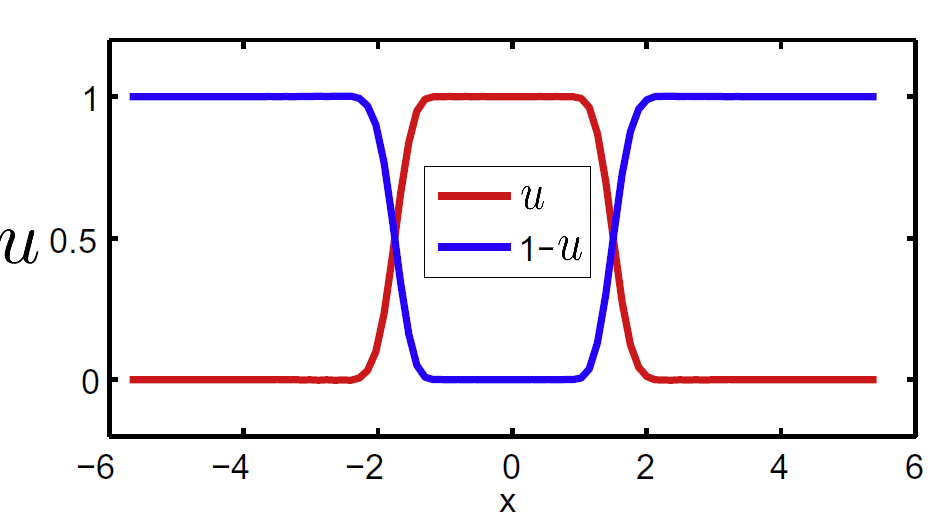}
  \caption{}
  %\label{fig:sub2}
\end{subfigure}
\caption{(A) Illustration of the model domain definition and decomposition: $\Omega_m$: solute (molecular) region; $\Omega_s$: solvent  region; $\Omega_t$: solute-solvent mixing region; (B) The cross line of $u$ and $(1 -u)$ of a diatomic system.}
\label{domain}
\end{figure}

%\textcolor{red}{(To Zhan: is this assumption reasonable? It excludes the possibility of a cavity. Disconnected $\Omega_t$ is admissible but complicated.)}
%By possibly shrinking $\Omega_m$ or enlarging $\Omega_s$, we can take $\Sigma_i$ ($i=0,1$) as smooth as we want.

%%%%%%%%%%%%%%%%%%%%%%%%%%%%%%%%%%%%%%%%%%%%%%%%%%%%%%%%%%%%

\subsection{A Novel Solvation Energy Functional}\label{Section 2.2}

%%%%%%%%%%%%%%%%%%%%%%%%%%%%%%%%%%%%%%%%%%%%%%%
%Nonpolar energy
%%%%%%%%%%%%%%%%%%%%%%%%%%%%%%%%%%%%%%%%%%%%%%%

As an improvement of the previous differential geometric based solvation model \cite{ZhanChen:2010a, Shao2022constrained}, we study a novel 
solvation free energy, whose nonpolar portion  is  defined as
\begin{equation}
\label{nonpolar energy}
\notag \Inp(u)=  \gamma \int_\Omega d|Du| + \int_\Omega \left[ P_h u^p  + \rho_s (1-u^p) U^{\mathrm{vdW}}   \right] \, dx  
\end{equation}
with $p=\frac{2N}{2N-1}$ for some integer $N>1$ and $\lambda, P_h>0$.  Note that $1<p<\frac{3}{2} $. Since $\frac{3}{2}=1^*$ is the Sobolev dual of $1$,  we have
$$
BV(\Omega)\hookrightarrow L^p(\Omega).
$$
Here $u:\Omega\to \R$ represents a characterizing function of the solute such that $u^p(x)$ is the volume ratio at position $x\in \Omega$ (as shown in Figure~\ref{domain}). As such, the physical constraints
\begin{equation}
\label{constrain 1}
u(x)\in [0,1] \quad \text{for a.a. } x\in \Omega
\end{equation}
and
\begin{equation}
\label{constrain 2}
u=1 \quad \text{a.e. in }\Omega_m  \quad \text{and} \quad  u=0 \quad \text{a.e. in }\Omega_s 
\end{equation}
need to be imposed.
Note   that $U^{\mathrm{vdW}}(x)$ can be formulated by $\sum_i U_i^{\rm att} (x)$ in which $U_i^{\rm att} (x)$ represents the attractive part of Lennard-Jones potential \cite{Wagoner:2006,ZhanChen:2010a}.  To this end, the L-J potential can be divided into attractive $U_i^{\rm att}$ and repulsive
$U_i^{\rm rep}$ in different ways. Here we take 
a Weeks-Chandler-Andersen (WCA) decomposition based on the original WCA theory \cite{Levy:2003}:
\begin{eqnarray}\label{potential3}
U_i^{\rm att,WCA}({\vec{r}})&=& \left\{
\begin{array}{ l l }\nonumber 
-\epsilon_{is}(x)  & 0 < \|x - x_i\|< 2^{1/6} \sigma_{is}
 \\
U_i^{\rm LJ}(x)~~~ & \|x - x_i\|\geq 2^{1/6}\sigma_{is},
\end{array} \right.\\
U_i^{\rm rep,WCA}(x)&=& \left\{
\begin{array}{ l l }\nonumber 
 U_i^{\rm LJ}(x) + \epsilon_{is}(x) ~~& 0< \|x - x_i\|< 2^{1/6}\sigma_{is}
 \\
0     & \|x - x_i\|\geq 2^{1/6} \sigma_{is}.
\\
\end{array} \right.
\end{eqnarray} 
where
$$
U_{i}^{\rm LJ}(r)=4\epsilon_{is} \left[ \left(\frac{\sigma_{is}}{r} \right)^{12} -  \left(\frac{\sigma_{is}}{r} \right)^{6} \right] 
$$	
with parameters $\epsilon_{is}$ of energy and $\sigma_{is}$ of length.

We choose $\Omega_m$ in such a way that there exist  balls $B(x_i,\tau)$ with $i=1,\cdots,N_a$ and $\tau>0$ such that 
\begin{equation}\label{ASP: atom supp}
%B(x_i,\tau)\cap B(x_j,\tau)=\emptyset \text{ for }i\neq j 
%\quad \text{and} \quad 
\bigcup\limits_{i=1}^{N_a} \overline{B}(x_i,\tau) \subset \Omega_m  
%\quad \text{and} \quad  U^{\rm vdW}  \leq 0 \quad \text{in } \Omega_w.
\end{equation}

%%%%%%%%%%%%%%%%%%%%%%%%%%%%%%%%%%%%%%%%%%%%%%%%%%%%%%%%%%%%
%polar energy functional
%%%%%%%%%%%%%%%%%%%%%%%%%%%%%%%%%%%%%%%%%%%%%%%%%%%%%%%%%%%%
The polar portion of the solvation free energy is defined as
\begin{align*}
%\label{polar energy functional 2}
\notag
\Ip (u,\psi) =   \int_\Omega &  \left[      \rho_m \psi    - \frac{1}{2}\epsilon(u) |\nabla \psi|^2 
-  (1- u^p )  B(\psi)     \right]\, dx.
\end{align*}
Here $\epsilon(u) = u^p   \epsilon_m + (1- u^p )  \epsilon_s$ is the dielectric constant of the solvent/solute mixture.
$\rho_m$ is supported in $\Omega_m$.
In addition, the neutral condition holds
\begin{equation}\label{neutral}
\sum\limits_{j=1}^{N_c} c_j^\infty q_j=0.
\end{equation}
Recall the definition of $B(\cdot)$ from \eqref{Def B}.
It is important to observe that $B(0)=0$  and, by \eqref{neutral}, $B^\prime(0)=0$ and $B^\prime(\pm \infty)=\pm \infty$. Further,  $B^{\prime\prime}(s)>0$. We thus conclude that $B(0)=\underset{{s\in \R}}{\min} B(s)$ and $B$ is strictly convex.

The problem of interest to us is to minimize the 
the total  energy functional
\begin{align}\label{non penalized total energy functional}
L(u,\psi)=   \Inp(u) + \Ip (u,\psi) ,
\end{align}
where $\psi$ satisfies the Dirichlet problem of a generalized Poisson-Boltzmann equation  
\begin{equation}
\label{GPBE}
\left\{\begin{aligned}
\div (\epsilon(u) \nabla \psi)  - (1-u^p) B'(\psi)&=  - \rho_m  &&\text{in}&&\Omega;\\
\psi&=\psi_\infty   &&\text{on}&&\partial\Omega
\end{aligned}\right.
\end{equation} 
for some
\begin{equation*}
%\label{bbdry cond}
\psi_\infty \in W^{1,\infty}(\Omega).
\end{equation*}
Therefore given   $u\in BV(\Omega)  $ satisfying \eqref{constrain 1}, $\psi=\psi(u)$ is determined via the elliptic boundary value problem~\eqref{GPBE}.

With the above observations, the minimization problem can be restated as to minimize
\begin{align}\label{non penalized total energy functional final form}
\hspace*{-2em}
I(u )  
= \gamma \int_\Omega d|Du| + \int_\Omega \left[ P_h u^p  + \rho_s (1-u^p) U^{\mathrm{vdW}} \right] \, dx  
+ \int_\Omega   \left[    \rho_m \psi    - \frac{1}{2}\epsilon(u) |\nabla \psi|^2 - (1-u^p)  B(\psi)         \right]\, dx
\end{align}
in the admissible space  
\begin{align*}
\cY=\{ & u\in BV(\Omega): \, u \, \text{ satisfies Constraints } \eqref{constrain 1} \text{ and }\eqref{constrain 2}\}  
\end{align*}
and $\psi=\psi(u)$ is determined via  \eqref{GPBE} in the space
\begin{equation*}
%\label{Def: cA}
\cA=\{v\in H^1(\Omega): \, v|_{\partial \Omega}=\psi_\infty\}.
\end{equation*}

%%%%%%%%%%%%%%%%%%%%%%%%%%%%%%%%%%%%%%%%%%%%%%%%%%%%%%%%%%%%%%
%\subsection{A Comparison with Sharp-Interface Solvation model}\label{Section 2.4}

%%%%%%%%%%%%%%%%%%%%%%%%%%%%%%%%%%%%%%%%%%%%%%%%%%%%%%%%

\section{A Family of Perturbed Poisson-Boltzmann Equation}\label{Section 3}

%%%%%%%%%%%%%%%%%%%%%%%%%%%%%%%%%%%%%%%%%%%%%%%%%%%%%%%%%%%%%%

%%%%%%%%%%%%%%%%%%%%%%%%%%%%%%%%%%%%%%%%%%%%%%%%%%%%%%%%%%%%%%

In this section, we study a sequence of functionals associated with the polar free energy, which will be used in the numerical simulations in Section~\ref{Section 6}.

Let $\{q_k\}_{k=1}^\infty$ be a sequence of decreasing real numbers with $\lim\limits_{k\to \infty} q_k= 1$  and taking values in  $\displaystyle \left(1,  \frac{\epsilon_s}{\epsilon_s-\epsilon_m}  \right)$. In addition, set $q_0=1$.
For any $u\in BV(\Omega)$ and $k=0,1,\cdots$, we put
\begin{align*}
G^k_u (\psi):= \int_\Omega  \left[     \frac{1}{2} \epsilon (u) |\nabla \psi|^2 -     \rho_m \psi +  (q_k-u^p )  B(\psi) \right]\, dx.
\end{align*}
Particularly, $G^0_u (\psi):= -\Ip(u, \psi)$.
Further, let $\cY_0=\cY$ and for $k=1,2,\cdots$ define
\begin{align}\label{Def: cYk}
  \cYk=\{   u\in W^{1,q_k}(\Omega): \,      | u |\leq \sqrt[p]{q_k}  \text{ a.e. in }\Omega \quad \text{and $u$ satisfies Constraint~\eqref{constrain 2}}\} .
\end{align}
Correspondingly, we introduce a sequence of perturbed Poisson-Boltzmann equations for $k=0,1,\cdots$
\begin{equation}
\label{GPBEk}
\left\{\begin{aligned}
\div (\epsilon(u) \nabla \psi)  - (q_k- u^p) B'(\psi)&=  -  \rho_m  &&\text{in}&&\Omega;\\
\psi&=\psi_\infty   &&\text{on}&&\partial\Omega.
\end{aligned}\right.
\end{equation}
%Note for any $u\in \cYk$, the dielectric constant $\epsilon(u)>0$.
In particular, when $k=0$, \eqref{GPBEk} coincides with \eqref{GPBE}.
Similar problems have been studied in \cite{MR3740372, MR2873259, MR3396402, Shao2022constrained}.

%%%%%%%%%%%%%%%%%%%%%%%%%%%%%%%%%%%%%%%%%%%%%
%Solution of GPBE
%%%%%%%%%%%%%%%%%%%%%%%%%%%%%%%%%%%%%%%%%%%%%% 
\begin{prop}\label{Prop: GPBE}
Given any $u\in \cYk$, $k=0,1,\cdots$, there exists a unique $\psi_u \in \cA$ such that
$$
G_u^k (\psi_u)=\underset{\psi\in \cA}{\min} G_u^k (\psi)<\infty.
$$
Moreover, $\psi_u$ is the unique weak solution to   \eqref{GPBEk}. Further, $\psi_u$ satisfies
\begin{align}
\label{est for psi 1}
\|\psi_u\|_{H^1} + \|\psi_u\|_\infty \leq \widetilde{C}_0 .
\end{align}
In particular, the constant $\widetilde{C}_0$   is independent of $\Omega_m$, $\Omega_s$, $u$ and $k$.
\end{prop}
\begin{proof}
Analogous problems have been studied in the literature on various Poisson-Boltzmann type equations, cf. \cite{MR3740372, MR2873259, MR3396402,  Shao2022constrained}.
In order to show the determining factors of the constant $\widetilde{C}_0$ in \eqref{est for psi 1},
we will, nevertheless, state a brief proof. 

For every $k$, $\epsilon(u)\in   L^\infty(\Omega)$ with 
$0<\epsilon_s- q_1 (\epsilon_s-\epsilon_m)\leq \epsilon(u)\leq \epsilon_s $. 
Standard elliptic theory, see\cite[Theorems 8.3 and 8.16]{MR737190}, implies that
\begin{equation}
\label{BVP}
\left\{\begin{aligned}
{\div} (\epsilon(u) \nabla \psi) +  \rho_m &=0  &&\text{in}&&\Omega;\\
\psi&=\psi_\infty   &&\text{on}&&\partial\Omega
\end{aligned}\right.
\end{equation}
has a unique weak solution $\hat{\psi}_u $, i.e. 
\begin{equation}
\label{weak sol 1}
\int_\Omega \epsilon(u)\nabla \hat{\psi}_u \cdot \nabla \phi\, dx =\int_\Omega  \rho_m \phi\, dx,\quad \forall  \phi \in H^1_0(\Omega),
\end{equation}
satisfying
\begin{align}\label{bound 1}
\|\hat{\psi}_u\|_{H^1}+\|\hat{\psi}_u\|_\infty   \leq M_0  .  
\end{align}
The constant $M_0$ depends only on $\Omega$, $\epsilon_s$, $\epsilon_m$, $q_1$ and $\psi_\infty$.
% and thus is independent of $u$.
Define $\tilde{G}^k_u: H^1_0(\Omega) \to \R \cup \{+\infty\}$ by
$$
\tilde{G}_u^k(\psi)= \int_\Omega  \left[     \frac{1}{2}\epsilon(u) |\nabla \psi|^2  +  (q_k-u^p)  B(\psi+\hat{\psi}_u)       \right] \, d x.
$$
By the direct method of calculus of variation and the strict convexity of  $\tilde{G}_u^k(\cdot)$, there exists a global minimizer $\bar{\psi}_u \in H^1_0(\Omega)$ of $\tilde{G}^k_u(\cdot)$.
\eqref{weak sol 1}  implies  
$$
G_u^k(\psi)=\tilde{G}_u^k(\psi-\hat{\psi}_u) + \int_\Omega \left[\frac{1}{2}\epsilon(u) |\nabla \hat{\psi}_u|^2  -  \rho_m \hat{\psi}_u \right]\, d x.
$$
Let $\psi_u=\hat{\psi}_u + \bar{\psi}_u$.
From the above equality, we learn that  $\psi_u$ minimizes $G_u^k(\cdot)$ in $\cY_k$.
Then following Steps (iii) and (iv) in the proof of  \cite[Proposition~2.2]{Shao2022constrained}, we can show that 
$$
\|\bar{\psi}_u\|_{\infty} +\|\bar{\psi}_u\|_{H^1} \leq M_1 
$$
for some constant $M_1$ depending only on $M_0$.
We can take $\widetilde{C}_0=M_0+M_1$.
\end{proof}
%%%%%%%%%%%%%%%%%%%%%%%%%%%%%%%%%%%%%%%%%%%%%
%Important est of G
%%%%%%%%%%%%%%%%%%%%%%%%%%%%%%%%%%%%%%%%%%%%%
The above proposition immediately gives the following crucial estimates.
For every $k$ and $u\in \cYk$,
\begin{align}
\notag
  G_u^k (\psi_u) < & G_u^k (\psi_\infty) 
=  \int_\Omega  \left[     \frac{1}{2}\epsilon(u) |\nabla \psi_\infty|^2 - \rho_m \psi_\infty + (q_k-u^p)   B(\psi_\infty)    \right]\, dx\\
\label{bound Eu}
\leq &   C \left[ \|\psi_\infty\|_{H^1}^2 + \|\psi_\infty\|_\infty +  B(\|\psi_\infty\|_\infty)   \right] 
\leq   \widetilde{C}_1 ,
\end{align}
where $\psi_u$ is the solution to \eqref{GPBEk}.
The constant $\widetilde{C}_1$ is independent of $\Omega_m$, $\Omega_s$, $k$ and the choice of $u$.

%%%%%%%%%%%%%%%%%%%%%%%%%%%%%%%%%%%%%%%%%%%%%
%Maximizer of polar energy
%%%%%%%%%%%%%%%%%%%%%%%%%%%%%%%%%%%%%%%%%%%%%%

\begin{prop}\label{Prop: GPBE converg}
Let $u_k\in \cYk$, $k=0,1,\cdots$, be such that
$$
  u_k \to u_0 \quad \text{in } L^1(\Omega) \quad \text{as } k\to \infty.
$$
Let $\psi_k  \in  \cA$ satisfy 
$ 
G_{u_k}^k (\psi_k)=\underset{w\in \cA}{\min} G_{u_k}^k (w)  .
$ 
Then 
\begin{equation}
\label{eq GPBEk converg 1}
\psi_k\to \psi_0 \quad \text{in }H^1(\Omega) \quad \text{and} \quad G_{u_k}^k (\psi_k) \to G_{u_0}^0(\psi_0)  \quad \text{as } k\to \infty.
\end{equation}
If, in addition,  $  u_k\in \cY$ and $\widetilde{\psi}_k  \in  \cA$ satisfies
$ 
G_{u_k}^0 (\widetilde{\psi}_k)=\underset{w\in \cA}{\min} G_{u_k}^0 (w)  .
$ 
Then 
\begin{equation}
\label{eq GPBEk converg 2}
\widetilde{\psi}_k\to \psi_0 \quad \text{in }H^1(\Omega) \quad \text{and} \quad G_{u_k}^0 (\widetilde{\psi}_k) \to G_{u_0}^0(\psi_0)  \quad \text{as } k\to \infty.
\end{equation}
%Then $\widetilde{\psi}_k\to \psi_0$ in $H^1(\Omega)$ and  $ G_{u_k}^0  (\widetilde{\psi}_k) \to G_{u_0}^0 (\psi_0)$ as $k\to \infty$.

\end{prop}
\begin{proof}
We will only prove \eqref{eq GPBEk converg 1}. 
The proof for \eqref{eq GPBEk converg 2} is similar.

Observe that since $u_k\to u_0 $ in $L^1(\Omega)$ and $\{ u_k \}_{k=0}^\infty$ are uniformly bounded in $L^\infty(\Omega)$.  
From the Riesz-Thorin interpolation theorem, we infer that $u_k\to u_0 $ in $L^r(\Omega)$ for all $r\in [1,\infty)$.
Further, by the mean value theorem
\begin{equation}
\label{u converg}
\lim\limits_{k\to \infty} \int |u_k^p-u_0^p|^r \, dx   \leq M \lim\limits_{k\to \infty}\| u_k - u_0\|_r^r =0 ,\quad r\in [1,\infty),
\end{equation}
for some constant $M>0$.

Due to \eqref{est for psi 1}, there exists a subsequence of $\{\psi_k\}_{k=1}^\infty$, not relabelled, and some $\psi\in H^1(\Omega)$ such that $\psi_k\to \psi$ in $L^2(\Omega)$ and $\psi_k \rightharpoonup \psi$ in $H^1(\Omega)$.
Since $\psi_k$ weakly solves \eqref{GPBEk} with $u=u_k$, for any $\phi\in C^1_0(\Omega)$
\begin{equation}
\label{weak formulation 1}
\int_\Omega \left[ \epsilon(u_k) \nabla \psi_k \cdot\nabla \phi + (q_k-u_k^p) B'(\psi_k) \phi \right]\, dx =\int_\Omega \rho_m \phi\, dx.
\end{equation}
The dominated convergence theorem then implies that
\begin{equation}
\label{weak formulation 2}
\int_\Omega \left[ \epsilon(u_0) \nabla \psi  \cdot\nabla \phi + (1-u_0^p) B'(\psi ) \phi \right]\, dx =\int_\Omega \rho_m \phi\, dx.
\end{equation}
Note that,   \eqref{est for psi 1} and a standard approximation argument imply that \eqref{weak formulation 1}  and \eqref{weak formulation 2} hold for any $\phi\in H^1_0(\Omega)$. 
In view of Proposition~\ref{Prop: GPBE}, we infer that $\psi_0=\psi$. Next, we will show that
\begin{equation}
\label{H1 converg}
\lim\limits_{k\to \infty} \int_\Omega \epsilon(u_k) |  \nabla \psi_k- \nabla \psi_0|^2\, dx=0.
\end{equation}
Using $\phi=\psi_k-\psi_0$ as a test function  in \eqref{weak formulation 1}, we conclude that
$$
\lim\limits_{k\to \infty} \int_\Omega \epsilon(u_k)   \nabla \psi_k \cdot (\nabla \psi_k - \nabla \psi_0)\, dx=0.
$$
By the dominated convergence theorem, we have
\begin{align*}
  \lim\limits_{k\to \infty} \int_\Omega \epsilon(u_k)  |\nabla \psi_k|^2\, dx  
=   \lim\limits_{k\to \infty} \int_\Omega \epsilon(u_k)  \nabla \psi_k \cdot (\nabla \psi_k- \nabla \psi_0)\, dx + \lim\limits_{k\to \infty} \int_\Omega \epsilon(u_k)  \nabla \psi_k \cdot   \nabla \psi_0 \, dx.
\end{align*}
Note that  $\psi=\psi_0-\psi_\infty$ weakly solves the Dirichlet problem
\begin{equation*}
%\label{GPBEk}
\left\{\begin{aligned}
\div (\epsilon(u_0) \nabla \psi)    &= (1- u_0^p) B'(\psi_0) -  \rho_m  -\div( \epsilon(u_0)\nabla \psi_\infty )   &&\text{in}&&\Omega;\\
\psi&=0   &&\text{on}&&\partial\Omega.
\end{aligned}\right.
\end{equation*} 
In view of  \eqref{est for psi 1},  $\epsilon(u_0)\nabla \psi_\infty$ and  $ (1- u_0^p) B'(\psi_0) -  \rho_m  $ belong to $  L^\infty(\Omega)$.
By the Calderon-Zygmund type estimates for uniformly elliptic equation, c.f. \cite[Theorem 1]{MR159110}, there exists   some 
$p_0>2$ 
such that $\psi_0\in W^{1,p_0}(\Omega)$. 
Note that \cite[Theorem 1]{MR159110} requires   $\Omega$ to be of class $\mathscr{D}^r$ for some $r>2$, cf. \cite[Formulas~(19) and (20)]{MR159110}. It follows from  \cite[Theorems~B and 3.1,   Lemma~4.1]{MR2141694} (by taking $T=\nabla (-\Delta)^{-1} \div $ in \cite[Theorem 3.1]{MR2141694}) and the Poincar\'e's inequality   that any Lipschitz domain satisfies this condition.
%Indeed, \cite[TLemma~4.1]{MR2141694}  shows that 
We thus infer from \eqref{u converg} that
\begin{equation}\label{H1 converg 2}
\lim\limits_{k\to \infty} \int_\Omega \epsilon(u_k)  \nabla \psi_k \cdot   \nabla \psi_0 \, dx=\int_\Omega \epsilon(u_0) | \nabla \psi_0 |^2\, dx,
\end{equation}
and in turn,
\begin{equation}\label{H1 converg 3}
\lim\limits_{k\to \infty} \int_\Omega \epsilon(u_k)  |\nabla \psi_k |^2\, dx= \int_\Omega \epsilon(u_0)  |\nabla \psi_0 |^2\, dx.
\end{equation}
The dominated convergence theorem, \eqref{H1 converg 2} and \eqref{H1 converg 3} imply that
\begin{align*}
  \lim\limits_{k\to \infty} \int_\Omega \epsilon(u_k) |  \nabla \psi_k- \nabla \psi_0|^2\, dx  
=  \lim\limits_{k\to \infty} \int_\Omega \epsilon(u_k) | \left( | \nabla \psi_k|^2 -2 \nabla \psi_k \cdot \nabla \psi_0 + | \nabla \psi_0|^2  \right)\, dx =0.
\end{align*}
This  establishes \eqref{H1 converg}. It follows from the Poincar\'e inequality that $\psi_k\to \psi_0$ in $H^1(\Omega)$.
The convergence $ G_{u_k}^k (\psi_k) \to G_{u_0}^0(\psi_0)$ then can be shown by using \eqref{H1 converg 3} and the dominated convergence theorem.
\end{proof}

%%%%%%%%%%%%%%%%%%%%%%%%%%%%%%%%%%%%%%%%%%%%%%%%%%%%%%%%%%%%%%
%\subsection{Existence of Minimizer: Total Solvation Free Energy}\label{Section 3.2}

%\begin{remark}\label{Rmk: exist Lip}
%The above theorem still holds true when $\Sigma_0$ and $\Sigma_1$ are Lipschitz.
%\end{remark}

%%%%%%%%%%%%%%%%%%%%%%%%%%%%%%%%%%%%%%%%%%%%%%%%%%%%%%%%%%%%

\section{Properties of Global Minimizers}\label{Section 4}

%%%%%%%%%%%%%%%%%%%%%%%%%%%%%%%%%%%%%%%%%%%%%%%%%%%%%%%%%%%%
%\subsection{Existence and Uniqueness of a Global Minimizer}\label{Section 4.1}
 
%%%%%%%%%%%%%%%%%%%%%%%%%%%%%%%%%%%%%%%%%%%%%%%%%%%%%%%%%%%%
%existence
%%%%%%%%%%%%%%%%%%%%%%%%%%%%%%%%%%%%%%%%%%%%%%%%%%%%%%%%%%%%
The following theorem on the  existence and uniqueness of a minimizer of $I(\cdot)$ can be proved essentially in the same way as  \cite[Theorem~2.4]{Shao2022constrained} by using Propositions~\ref{Prop: GPBE converg}, \ref{A1: compact of BV} and \ref{A1: lsc of BV}.
\begin{theorem}\label{Thm: existence of global minimizer}
There exists a unique  $\um \in \cY$ such that $I(\um)=\underset{{u\in \cY}}{\min} I(u)$.
\end{theorem}

%%%%%%%%%%%%%%%%%%%%%%%%%%%%%%%%%%%%%%%%%%%%%%%%%%%%%%%%%%%%
%\subsection{Continuous Dependence on $\Omega_m$ and $\Omega_s$}\label{Section 4.2}
%%%%%%%%%%%%%%%%%%%%%%%%%%%%%%%%%%%%%%%%%%%%%%%%%%%%%%%%%%%%
%free energy onvergence
%%%%%%%%%%%%%%%%%%%%%%%%%%%%%%%%%%%%%%%%%%%%%%%%%%%%%%%%%%%%
To show the robustness of the model~\eqref{non penalized total energy functional final form}, one need to answer the question whether the   solvation energy $I(\um)$  depends continuously on $\Omega_m$ and $\Omega_s$ in a suitable topology?
The answer to the above question is affirmative.
We will present the proof of a partial result in this subsection. 
Due to the length of this article,  a complete answer will be presented  in a subsequent paper.

%The robustness of the model~\eqref{non penalized total energy functional final form} is closed associated with the following question. If we consider the predicted solvation energy $I(\um)$  as a function of $\Omega_m$ and $\Omega_s$, does $I(\um)$ depend  continuously on $\Omega_m$ and $\Omega_s$ in a suitable topology? 
  
%%%%%%%%%%%%%%%%%%%%%%%%%%%%%%%%%%%%%%%%
%approx sequence domain
%%%%%%%%%%%%%%%%%%%%%%%%%%%%%%%%%%%%%%%%%

Assume that $\{ \widetilde{\Omega}_{m;n} \}_{n=1}^\infty $ and $\{ \widetilde{\Omega}_{s;n} \}_{n=1}^\infty$
%, satisfying  $\partial\Omega\subset \partial\widetilde{\Omega}_{s;n}$, 
are two sequences of Lipschitz subdomains  such that 
\begin{equation}
\label{sequence domain}
\bigcup\limits_i^{N_a} \overline{B}(x_i,\sigma)\subset \widetilde{\Omega}_{m;n} \subseteq \Omega_m \quad \text{and} \quad \widetilde{\Omega}_{s;n} \subseteq \Omega_s \quad \text{with }\partial\Omega\subset \partial\widetilde{\Omega}_{s;n}.
\end{equation}
We consider the sequence of energy functionals $\widetilde{I}_n (\cdot)$ defined by replacing $\Omega_m$ and $\Omega_s$ by  $\widetilde{\Omega}_{m;n} $ and $\widetilde{\Omega}_{s;n}$ in $I(\cdot)$,  respectively. 
The corresponding admissible spaces are
\begin{align*}
\widetilde{\cY}_n=\{   u\in BV(\Omega): \,   0 \leq u \leq 1 \text{ a.e. in }\Omega \quad \text{and}\quad u=1\text{ a.e. in }\widetilde{\Omega}_{m;n}  
 \text{ and } u=0\text{ a.e. in }\widetilde{\Omega}_{s;n}\} .
\end{align*}

\begin{theorem}\label{Thm: free energy converg}
Assume \eqref{sequence domain} and  as $n\to \infty$
\begin{equation}
\label{eq: free energy converg 1}
\chi_{\widetilde{\Omega}_{m;n}}\to \chi_{\Omega_m} \quad \text{and} \quad \chi_{\widetilde{\Omega}_{s;n}}\to \chi_{\Omega_s} \quad \text{in } L^1(\Omega).
\end{equation}
Then for each $n$, there is a unique minimizer $u_n$ of $\widetilde{I}_n (\cdot)$ in $\widetilde{\cY}_n$.
Moreover, 
 $\lim\limits_{n\to \infty}  \widetilde{I}_n (u_n) =I(\um)$.	 
\end{theorem}
\begin{proof}
The existence and uniqueness of a minimizer of $\widetilde{I}_n(\cdot)$ in $\widetilde{\cY}_n$ for each $n$ follows from Theorem~\ref{Thm: existence of global minimizer}.
%By Theorem~\ref{Thm: existence of global minimizer}, for each $n$, there is a unique minimizer $u_n$ of $\widetilde{I}_n (\cdot)$ in $\widetilde{\cY}_n$. 
%We denote   by $\psi_n=\psi_{u_n}$, the solution of \eqref{GPBE} with $u=u_n$.
Observe that $\um \in \widetilde{\cY}_n$ for all $n$. Thus
$$
\widetilde{I}_n (u_n) \leq  I(\um)=\widetilde{I}_n(\um).
$$
This implies that
\begin{align*}
\gamma\int_\Omega  d|D u_n| +    P_h \|u_n\|_p^p + \rho_s  \int_{\Omega\setminus \Omega_m}   U^{\rm vdW}  \, dx  -\widetilde{C}_1 \leq I(\um),
\end{align*}
where $\widetilde{C}_1$ is the constant in \eqref{bound Eu}. 
Therefore, $\| u_n\|_{BV}$ is uniformly bounded with respect to $n$.
Proposition~\ref{A1: compact of BV} implies that 
there exists a  subsequence, not relabelled, and some $u\in BV(\Omega)$  such that $u_n\to u$ in $L^1(\Omega)$. 
%In view of \eqref{eq: free energy converg 1}, we conclude that $u=\chi_{\Omega_i}$. 
From  Propositions~\ref{A1: lsc of BV}, Propositions~\ref{Prop: GPBE converg} and the dominated convergence theorem, we infer that
\begin{align*}
I(\um) \leq I(u)\leq \liminf\limits_{n\to \infty} \widetilde{I}_n (u_n) \leq \limsup\limits_{n\to \infty} \widetilde{I}_n (u_n) \leq I(\um).
\end{align*}
This proves the convergence assertion.
\end{proof}

A case of particular interest is when $\Omega_t=\emptyset$, that is, $\Omega=\Omega_m \cup \Gamma  \cup\Omega_s$ with   $\Gamma=\partial\Omega_m\cap \partial \Omega_s$ being the Lipschitz sharp interface separating the solute and solvent regions.
Further, suppose that $\Omega_m\subset\subset \Omega$.
In this case, \eqref{non penalized total energy functional final form} reduces to a sharp interface model.  
The corresponding sharp-interface solvation free energy   is given by the one  proposed in 
\cite{PhysRevLett.96.087802, doi:10.1063/1.2171192}  
\begin{align}\label{Sharp interface energy functional}
 E_0  
=&   \gamma \Per(\Omega_m;\Omega)+ P_h \cL^3(\Omega_m)   + \int_{\Omega_s}     \rho_s U^{\mathrm{vdW}}  \, dx + G_{\rm ele}(\Omega_m),
\end{align}
where $\Per(\Omega_m; \Omega)$ is the perimeter  of $\Omega_m$ in $\Omega$, see Appendix~\ref{Appendix A}, and 
$G_{\rm ele}(\Omega_m)$ is the electrostatic free energy.  
In the classic Poisson-Boltzmann theory, it is defined by 
\begin{align*}
G_{\rm ele}(\Omega_m)=      \int_{\Omega_m}   \left[     \rho_m \psi    - \frac{\epsilon_m}{2} |\nabla \psi|^2 \right]\, dx  
  - \int_{\Omega_s} \left[ \frac{\epsilon_s}{2} |\nabla \psi|^2  +   B(\psi)    \right]\, dx,
\end{align*}
cf. \cite{ANDELMAN1995603, doi:10.1021/jp7101012, doi:10.1021/cr00101a005, MR2491589, doi:10.1146/annurev.bb.19.060190.001505, doi:10.1063/1.466406, doi:10.1021/ct401058w}.
The electrostatic potential
$\psi $ solves the classic sharp-interface Poisson-Boltzmann equation:
\begin{equation*}
%\label{GPBES}
\left\{\begin{aligned}
\div ( (\epsilon_m \chi_{\Omega_m}+\epsilon_s \chi_{\Omega_s}) \nabla \psi)  -\chi_{\Omega_s} B'(\psi)&=  -  \rho_m  &&\text{in}&&\Omega;\\
\psi&= \psi_\infty   &&\text{on}&&\partial\Omega.
\end{aligned}\right.
\end{equation*}

The following corollary shows that \eqref{Sharp interface energy functional} is in some sense the limiting case of our diffuse interface model.
\begin{cor}
Assume that $\Omega=\Omega_m \cup \Gamma  \cup\Omega_s$ and $\Gamma=\partial\Omega_m\cap \partial \Omega_s$ is Lipschitz.
Further, suppose that $\Omega_m\subset\subset \Omega$.
Under the same assumptions as in Theorem~\ref{Thm: free energy converg}, 
$\lim\limits_{n\to \infty}  \widetilde{I}_n (u_n) =E_0.$ 
\end{cor}
 
\medskip
\begin{remark}
In a subsequent paper,  we will show that, under mild regularity assumption on $\Sigma_1$ and $\Sigma_0$, the conditions $\widetilde{\Omega}_{m;n} \subseteq \Omega_m$ and $\widetilde{\Omega}_{s;n} \subseteq \Omega_s$ in Theorem~\ref{Thm: free energy converg} can be relaxed.
\end{remark}

%%%%%%%%%%%%%%%%%%%%%%%%%%%%%%%%%%%%%%%%%%%%%%%%%%%%%%%%%%%%

%%%%%%%%%%%%%%%%%%%%%%%%%%%%%%%%%%%%%%%%%%%%%%%%%%%%%%%%%%%%
\section{How to Exclude the Formation of Sharp Interfaces?}\label{Section 5}

In Theorem~\ref{Thm: existence of global minimizer}, we have shown that there is a unique characterizing function $\um\in BV(\Omega)$ minimizing \eqref{non penalized total energy functional final form} in $\cY$.
However, since $BV-$functions allow jump discontinuities, a natural question to ask is whether the minimizing energy state is achieved by a sharp interface between the solute and solvent regions, or equivalently, whether the characterizing function $\um$ is the characteristic function of a set of finite perimeter. 
%An even more profound question is whether the interface is indeed ``smooth", or equivalent whether the minimizer $u$ is a continuous function.

To simplify the analysis, we will focus on the nonpolar portion of the solvation energy, i.e. \eqref{nonpolar energy}.
Motived by  the idea in \cite{MR3148636, MR2815736, MR2368971}, we will show  that when the mean curvature of $\Sigma_0$ is positive at some point, the energy minimizing state is never achieved by a sharp interface. See Theorem~\ref{Thm: sharp interface}. 
%The idea of our proof is motived by \cite{MR3148636, MR2815736, MR2368971}.

%the only possible positions of discontinuities of a minimizer $u$ are $\Sigma_0 \cup \Sigma_1$, c.f. Theorem~\ref{Thm: regularity of minimizer}. In other words, if a sharp interface is ever formed to separate the solute and the solvent regions, it can only be $\Sigma_0$ or $\Sigma_1$. \textcolor{blue}{More details...} One can easily verify whether SES or SAS minimizes the solvation energy by comparing them with several choices of $u$.  What is more, we have proved that the minimizer $u$ must be continuous between SES and SAS, c.f. Remark~\ref{Rmk: continuity of u}. 
%Our work  partially answers the above questions.

%%%%%%%%%%%%%%%%%%%%%%%%%%%%%%%%%%%%%%%%%%%%%%%%%%%%%%%%%%%%
\subsection{Necessary Conditions for the Minimizer of Nonpolar Energy}\label{Section 5.1}
%%%%%%%%%%%%%%%%%%%%%%%%%%%%%%%%%%%%%%%%%%%%%%%%%%%%%%%%%%%%
%Existence of nonpolar minimizer
%%%%%%%%%%%%%%%%%%%%%%%%%%%%%%%%%%%%%%%%%%%%%%%%%%%%%%%%%%%%
Throughout this section, we assume that $\Omega_t \neq \emptyset$.
First consider the minimization problem of the nonpolar energy
\begin{align}\label{minimization pb nonpolar 1}
\Inp(u) =& \gamma \int_\Omega d|Du| + \int_\Omega \left[ P_h u  + \rho_s (1-u^p) U^{\mathrm{vdW}}   \right] \, dx   
\end{align}
in the admissible space
$$
\cX=\{u\in BV(\Omega): \, u \, \text{ satisfies Constraint  } \eqref{constrain 2}  \}.
$$
One will show that the  minimizer $\um$ of  \eqref{minimization pb nonpolar 1} automatically satisfies Constraint  \eqref{constrain 1}.
The reason to exclude \eqref{constrain 1} in the definition of the admissible space is due to the following consideration.
Any subdifferential of $\Inp(\cdot)$ with Constraint~\eqref{constrain 1} contains a function which may be discontinuous along $\partial \{\um=1\}$ and $\partial \{\um=0\}$. 
This will prevent us from establishing the continuity of $\um$ in these two sets.

\medskip
\begin{theorem}\label{Thm: existence and uniqueness minimizer - nonpolar}
\eqref{minimization pb nonpolar 1}  has a unique minimizer $\um\in \cX$, which satisfies Constraint~\eqref{constrain 1}. 
\end{theorem}
\begin{proof}
Note that $\cX$ is closed and convex in $BV(\Omega)$.
Based on the strict convexity, lower semicontinuity of $\Inp$ and the direct method of Calculus of Variation, we can readily establish the existence and uniqueness of a global minimizer $\um$.
If $\cL^3(\{\um>1\}\cup \{\um<0\})>0$, let
\begin{align*}
%\label{truncation 0}
\widetilde{u}_{\rm min}(x)=
\begin{cases}
1 \quad &\text{when } \um (x)>1;\\
0 &\text{when } \um(x)<0;\\ 
\um(x) & \text{elsewhere}.
\end{cases}
\end{align*}
Direct computations show that $\Inp(\widetilde{u}_{\rm min}) < \Inp (\um)$. A contradiction. Therefore, $0\leq \um\leq 1$ a.e. in $\Omega$.
\end{proof}

%%%%%%%%%%%%%%%%%%%%%%%%%%%%%%%%%%%%%%%%%%%%%%%%%%%%%%%%%%%%
%equivalent pb
%%%%%%%%%%%%%%%%%%%%%%%%%%%%%%%%%%%%%%%%%%%%%%%%%%%%%%%%%%%%

Next, we   derive  necessary conditions for the minimizer of \eqref{minimization pb nonpolar 1}.
We will use tools from non-smooth analysis, c.f. \cite{MR1727362, MR1313551, MR1637956}, to derive the subdifferential of \eqref{minimization pb nonpolar 1}. 
However,  very little is known about the dual space of  $BV(\Omega)$. 
To overcome this difficulty and tackle the Constraint~\eqref{constrain 2},  we will   consider  $\Inp$ as a functional defined on  $L^p(\Omega)$ and include two  extra terms.
Define
\begin{equation}
\label{minimization pb nonpolar 2}
\Enp(u)= \Inp (u) + \gamma \int_{\partial\Omega} | \Tr u|\, d\cH^2 + I_K(u)  
\end{equation}
in $L^p(\Omega)$, where $\Tr u$ is the trace of $u$ on $\partial\Omega$ and 
$$
K=\{u\in L^p(\Omega):  \,    u=1\text{ in }\Omega_m,   
 \text{ and } u=0\text{  in } \Omega_s  \text{ a.e.} \}
$$
and $I_K$ is the indicator function of $K$. In addition, we put
$$
E_1(u)=\gamma \|D u\|(\Omega) +\gamma \int_{\partial\Omega} | \Tr u|\, d\cH^2  , 
$$
and
$$
E_2(u)= \int_{\Omega} \left[ P_h u^p  + \rho_s (1-u^p) U^{\mathrm{vdW}} \right] \, dx.
$$
The latter is Lipschitz continuous in $L^p(\Omega)$. 
It is understood that
\begin{align*}
E_1(u)  =
\begin{cases}
\gamma \|D u\|(\Omega) +\gamma \int_{\partial\Omega} | \Tr u|\, d\cH^2   \quad & \text{if }u\in    BV(\Omega)  \\
+\infty & \text{if } u\in L^p(\Omega) \setminus  BV(\Omega) .
\end{cases}
\end{align*}
So, $\dom(E_1)=  BV(\Omega)$ and $\dom (I_K)=K$.
Using these notations, we can restate Problem~\eqref{minimization pb nonpolar 2} as to minimize a functional $\Enp: L^p(\Omega) \to \R\cup \{\infty\}$ defined by
\begin{equation}
\label{miimizing unconstrained pb}
\Enp(u):=E_1(u) + E_2(u)+I_K(u).
\end{equation}
Direct computations show that $\um$ minimizes \eqref{minimization pb nonpolar 1} in $\cX$ iff it minimizes
$\Enp(\cdot)$ in $L^p(\Omega)$.

%%%%%%%%%%%%%%%%%%%%%%%%%%%%%%%%%%%%%%%%%%%%%%%%%%%%%%%%%%%%
%necessary cond of nonpolar
%%%%%%%%%%%%%%%%%%%%%%%%%%%%%%%%%%%%%%%%%%%%%%%%%%%%%%%%%%%%

%%%%%%%%%%%%%%%%%%%%%%%%%%
%subdiff-I_K
%%%%%%%%%%%%%%%%%%%%%%%%%%
Note that   $K$ is closed and convex in $L^p(\Omega)$. This  implies that $I_K$ is convex and lower semicontinuous. 
What is more, by the definition of subdifferentials, for every $u\in K$, $u^*\in \partial I_K (u)$ iff
$$
\langle u^*, u \rangle \geq \langle u^*, v \rangle ,\quad \forall v \in K.
$$
Here  $\langle \cdot,\cdot \rangle$ is the duality pairing between $L^p(\Omega)$ and $L^{p'}(\Omega)$, that is
$$
\langle v,w \rangle=\int_\Omega vw \, dx ,\quad v \in L^p(\Omega),\, w\in L^{p'}(\Omega).
$$
If $\cL^3(\{u^*> 0\}\cap \Omega_t)>0$, set $D=\{u^*> 0\}\cap \Omega_t$.
We define
\begin{align*}
v(x)=
\begin{cases}
u(x)+1 ,\quad & x\in D\\
u(x), & \text{elsewhere}.
\end{cases}
\end{align*}
Then $v\in K$  and
$$
\langle u^*, v-u \rangle >0.
$$
A contradiction. Similarly, we can show that $\cL^3(\{u^*< 0\}\cap \Omega_t)=0$. Thus, $u^*=0$ a.e. in $\Omega_t$. This is also the sufficient condition of $u^*\in \partial I_K (u)$. 
Indeed, given any $u^*\in L^{p'}(\Omega)$ with $u^*=0$ a.e. in $\Omega_t$, for any $v\in K$,
$$
\langle u^*, v-u \rangle = \int_{\Omega\setminus \Omega_t} u^*(u-v)\, dx + \int_{\Omega_t} u^*(u-v)\, dx = 0.
$$
To sum up, a function $u^*\in L^{p'}(\Omega)$ belongs to $\partial I_K(u)$ iff  $u^*=0$ in $ \Omega_t$.

%%%%%%%%%%%%%%%%%%%%%%%%%%
%subdiff-E_1
%%%%%%%%%%%%%%%%%%%%%%%%%%
To compute $\partial E_1(u)$,  we define 
$$
X^\infty_{p'}:=\{z\in L^\infty(\Omega, \R^3):\, \div z \in L^{p'}(\Omega)\} .
$$
Here, $\div z\in L^{p'}(\Omega)$ means that there exists $f\in L^{p'}(\Omega)$ such that 
$$
\int_\Omega f \phi\, dx = -\int_\Omega z\cdot \nabla \phi \, dx
$$
for all $\phi\in C_0^\infty(\Omega)$. Given any $u\in BV(\Omega) $ and $z\in X^\infty_{p'}$, there exists a Radon measure, denoted by $z\cdot Du$, such that for any $\phi\in C_0^\infty(\Omega)$, with a little abuse of notation,
$$
\langle z\cdot Du, \phi \rangle := \int_\Omega \phi( z\cdot Du) = -\int_\Omega u z \div\phi\, dx - \int_\Omega u\phi \div z\, dx.
$$
The measure $z\cdot Du$ is absolutely continuous with respect to $|Du|$. By the Radon-Nikodym Theorem, there is a $|Du|$-measurable function $\theta(z,Du)$ s.t. 
\begin{equation}
\label{density}
\int_A z\cdot Du =\int_A \theta(z,Du) d|Du|
\end{equation}
for all Borel sets $A\subseteq \Omega$. 
Let
\begin{equation*}
%\label{Def: M*}
M^*_{p'}:=\{v^*\in L^{p'}(\Omega): \,  v^*= - \div z \text{ for some } z\in X^\infty_{p'} \text{ with } \|z\|_\infty \leq 1\}.
\end{equation*}
One can  follow the idea of \cite[Proposition~4.23(1)]{MR2348842} and prove that 
$$
u^*\in \partial E_1(u)\quad \text{iff} \quad E_1(u)=\gamma \langle u^*, u \rangle   ,\quad u^*\in M^*_{p'},
$$
that is,
\begin{equation}
\label{subdiff of E1}
E_1(u)= - \gamma \int_\Omega u \div z \, dx = \gamma \int_\Omega z\cdot Du  -\gamma	\int_{\partial\Omega} (z\cdot \nu_{\partial\Omega}) u \, d\cH^2 
\end{equation}
for some $z\in X^\infty_{p'} $ with $ \|z\|_\infty \leq 1$, where $\nu_{\partial\Omega}$ is the outward unit normal of $\Omega$.
The last equality follows from \cite[Theorem~1.9]{MR750538}.
%Moreover, if $E_1(v)\neq 0$, $\|z\|_\infty=1$.
%It state that every selection $v^*\in \partial E_1(v)$ can be written as
%$$
%v^*=-\div z \quad \text{for some } z\in X^\infty_2 
%$$
%such that
%$$
%E_1(v)= \gamma  \int_\Omega v \div z \, dx .
%$$
In addition,  \cite[Corollary~1.6]{MR750538} shows that $\|z\|_\infty=1$ whenever $u\neq 0$.

%%%%%%%%%%%%%%%%%%%%%%%%%%
%subdiff-E_2  
%%%%%%%%%%%%%%%%%%%%%%%%%%
%%%%%%%%%%%%%%%%%%%%%%%%%%
%subdiff-K_2  
%%%%%%%%%%%%%%%%%%%%%%%%%%
Next, Proposition~\ref{Prop: Frechet} implies that for any $u\in L^p(\Omega)$, 
$$
\partial E_2(u)=p P_h u^{p-1}  - p\rho_s u^{p-1} U^{\rm vdW}.
$$
%%%%%%%%%%%%%%%%%%%%%%%%%%
%linear space cond
%%%%%%%%%%%%%%%%%%%%%%%%%%
Because of the lack of continuity of $E_1$ and $I_K$, in general, we can only conclude that $\partial  E_1(u) + \partial I_K(u) \subseteq \partial (E_1+I_K)(u)$.
In order to compute $\partial (E_1+I_K)(u)$, we will use Propositions~\ref{Prop: sum rule of subdifferential 2}.
It suffices to verify the closed linear space  condition. 
An easy computation shows that 
$$
\dom(E_1)-\dom(I_K)=\{v\in L^p(\Omega):\, v|_{\Omega_m\cup \Omega_s}\in BV(\Omega_m\cup \Omega_s)\},
$$
which is obviously a linear subspace of $L^p(\Omega)$. We learn from Propositions~\ref{A1: lsc of BV} and \ref{A1: extension of BV} that $\dom(E_1)-\dom(I_K)$ is closed.
Now Proposition~\ref{Prop: sum rule of subdifferential 2} immediately implies that
$$
\partial (E_1+I_K)(u)=\partial E_1 (u)+ \partial I_K(u).
$$
We thus have
\begin{equation}
\label{eq: sum of subdiff 2}
\partial E (u)=\partial E_1 (u)+\partial E_2 (u)+\partial I_K(u).
\end{equation}
%for the minimizer $u$ of \eqref{minimization pb nonpolar 2} .

%%%%%%%%%%%%%%%%%%%%%%%%%%
%A necessary cond for min
%%%%%%%%%%%%%%%%%%%%%%%%%%
From the definition of subdifferential and \eqref{eq: sum of subdiff 2}, we learn that  
$$
u \in \cX \text{ minimizes }\eqref{minimization pb nonpolar 2} \quad \text{iff} \quad  0\in \partial E (u)=\partial E_1 (u)+\partial E_2 (u)+\partial I_K (u).
$$
More precisely, this means that there is some $z\in X^\infty_{p'}$ with $\|z\|_\infty = 1$  satisfying \eqref{subdiff of E1} and $w\in L^{p'}(\Omega)$ with $w\equiv 0$ in $\Omega_t $   such that
\begin{equation}
\label{necessary cond minimizer nonpolar-0}
0= -\gamma\div z + p \um^{p-1} \left( P_h - \rho_s  U^{\mathrm{vdW}} \right) + w   \quad \text{in } \Omega,
\end{equation}
where $z$ satisfies
$$
\int_\Omega z\cdot D\um= -\int_\Omega  \um \div z\, dx = \|D \um\|(\Omega)  .
$$
In particular, it holds that
\begin{equation*}
%\label{necessary cond minimizer nonpolar}
0= -\gamma\div z +p \um^{p-1} \left( P_h - \rho_s  U^{\mathrm{vdW}} \right)   \quad \text{in } \Omega_t .
\end{equation*}

%%%%%%%%%%%%%%%%%%%%%%%%%%%%%%%%%%%%%%%%%%%%%%%%%%%%%%%%%%%%
\subsection{Regularity of the Minimizer $\um$}\label{Section 5.2}

As in the previous subsection, $\um$ is the minimizer of \eqref{minimization pb nonpolar 2} in $L^p(\Omega)$.
Set 
\begin{equation}
\label{super-level}
E_t:=\{\um > t\}, \quad t\in [0,1)
\end{equation}
to be the super-level sets of $\um$. Recall $\Omega_w=\Omega\setminus \overline{\Omega}_s$.
\begin{prop}
\label{Prop: nonpolar to finite perimeter}
For all $t\in [0,1)$, $E_t$ is a solution of 
\begin{equation}
\label{Mini pb: finite perimeter}
 \underset{E \in \cM}{\min} \left[ \gamma \Per(E;\Omega)  +   \int_{E } pt^{p-1 }\left(  P_h-  \rho_s U^{\mathrm{vdW}}     \right) \, dx \right],
\end{equation}
where   the minimum is taken in the set
$$
\cM=\{E\subset \Omega \text{ is of finite perimeter} : \, \Omega_m\subseteq E  \subseteq \Omega_w\}.
$$
%\textcolor{red}{Double check! I think we need to change $\int_E t^{p-1} \rho_s  U^{\mathrm{vdW}}\, dx $ to $\int_{E\setminus \Omega_m} t^{p-1} \rho_s  U^{\mathrm{vdW}}\, dx $.}
\end{prop}
\begin{proof}
Take $z$ as in \eqref{necessary cond minimizer nonpolar-0}.
\eqref{density} and \eqref{subdiff of E1} show that
$$
\|D\um\|(\Omega)=\int_\Omega z\cdot D\um=\int_\Omega \theta(z,D\um) d|D\um|.
$$
By \cite[Corollary 1.6]{MR750538}, it holds that $\|\theta(z,D\um)\|_{L^\infty(\Omega,|D\um|)}\leq \|z\|_\infty =1$.
We thus infer that $\theta(z,D\um)=1$ $|D\um|$-a.e. 
For any $a,b\in [0,1)$ with $a<b$, define
\begin{align*}
v(x)=
\begin{cases}
b \quad &\text{if }\um(x)> b  \\
\um(x)  &\text{if }a	\leq \um(x)\leq b\\
a &\text{if } \um(x)<a.
\end{cases}
\end{align*}
%Then it follows from \cite[Formula~(2.15)]{MR750538} that
%$$
%\theta(z,D\um)=1=\theta(z,Dv) \quad |Dv|-\text{a.e. in }\Omega.
%$$
%by observing that $|D\um|=|D v|$ on the support of $|Dv|$ in the sense of measure. (Double check!)
Given any $\phi\in C_0(\Omega)$, by  \cite[Proposition~2.7(i) and Formula~(2.15)]{MR750538}, we have
\begin{align*}
\int_\Omega \phi \, d|D v|=  \int_\Omega \phi \theta(z,Dv)  d|D v|   
= \langle z\cdot Dv,\phi \rangle   
=  \int_a^b \int_{\Omega } \phi (z\cdot  D\chi_{E_t}) \, dt .
\end{align*}
On the other hand, by the coarea formula~\eqref{coarea formula},
\begin{align*}
\int_\Omega \phi \, d|D v|  
=  \int_a^b \int_{\Omega } \phi \, d|D\chi_{E_t}|  \, dt .
\end{align*}
It shows that
$$
\int_a^b \int_{\Omega } \phi (z\cdot  D\chi_{E_t}) \, dt=\int_a^b \int_{\Omega } \phi \, d|D\chi_{E_t}|  \, dt ,\quad \forall \phi\in C^\infty_0(\Omega).
$$
Because $a$ and $b$ are arbitrary,  $(z\cdot  D\chi_{E_t}) =|D\chi_{E_t}| $ in the sense of measure for a.a. $t\in [0,1)$.
This implies that
\begin{equation}
\label{um-super level}
\int_\Omega  z \cdot D\chi_{E_t}  =\|D\chi_{E_t}\|(\Omega) \quad \text{for a.a. }t\in [0,1).
\end{equation}
Denote by $D$ the set of all  $t$ satisfying \eqref{um-super level}.
If $t\in D$, \eqref{um-super level} and \cite[Corollary~1.6, Theorem~1.9]{MR750538} imply that
\begin{align*}
-\int_\Omega \div z (\chi_F -\chi_{E_t})\, dx = & \int_\Omega  z \cdot D\chi_F \, dx - \int_\Omega  z \cdot D\chi_{E_t} \, dx  
=  \int_\Omega  z \cdot D\chi_F \, dx - \Per(E_t;\Omega) \\
\leq & \Per(F;\Omega)-\Per(E_t;\Omega) 
\end{align*}
holds for all $F \in \cM$.
Combining with 
\eqref{necessary cond minimizer nonpolar-0}, we thus deduce that
\begin{align*}
&\gamma\Per(F;\Omega)-\gamma\Per(E_t;\Omega) \\
\geq & -\int_{\Omega} p \um^{p-1} \left( P_h-  \rho_s U^{\mathrm{vdW}} \right)(\chi_F -\chi_{E_t})\, dx - \int_\Omega w (\chi_F -\chi_{E_t})\, dx \\
\geq & -\int_{\Omega} pt ^{p-1 } \left( P_h-  \rho_s U^{\mathrm{vdW}} \right) (\chi_F -\chi_{E_t})\, dx   \\
& +  \int_\Omega  p(t^{p-1}-\um^{p-1})\left( P_h-  \rho_s U^{\mathrm{vdW}} \right)(\chi_F -\chi_{E_t})\, dx \\
\geq   &  -\int_{\Omega}   pt ^{p-1 } \left(  P_h- \rho_s U^{\mathrm{vdW}}   \right)(\chi_F -\chi_{E_t})\, dx
\end{align*}
by observing that
$$
(t^{p-1}-\um^{p-1})(  P_h- \rho_s U^{\mathrm{vdW}}  )(\chi_F -\chi_{E_t}) \geq  0 
$$
and
$$
\int_\Omega w (\chi_F -\chi_{E_t})\, dx = 0.
$$
If $t\notin D$, then take a decreasing sequence $\{t_n\}_{n=1}^\infty \subset D$ such that $t_n\to t^+$. 
It is clear that $\bigcup\limits_n E_{t_n}=E_t$. By the dominated convergence theorem, $\chi_{E_{t_n}}\to \chi_{E_t}$ in $L^1(\Omega)$. Then Proposition~\ref{A1: lsc of BV} shows that
%$$
%\int_{E_{t_n} } pt_n^{p-1 }\left(  P_h-  \rho_s U^{\mathrm{vdW}}     \right) \, dx \to \int_{E_t } pt ^{p-1 }\left(  P_h- pt ^{p-1 }\rho_s U^{\mathrm{vdW}}     \right) \, dx
%$$
%as $n\to \infty$ and 
$$
\Per(E_t;\Omega) \leq \liminf\limits_{n\to \infty} \Per(E_{t_n};\Omega).
$$
%For any $\varepsilon>0$, pick $A\subset \Omega$ such that
%$$
%\|D \chi_{E_t}\|(\Omega \setminus A) <\varepsilon.
%$$
%Take $\phi\in C^\infty_0(\Omega,[0,1])$ such that $\phi\equiv 1$ on $A$.
On the other hand, \eqref{um-super level} and \cite[Corollary~1.6 and Theorem~1.9]{MR750538} imply that
\begin{align*}
\Per(E_{t_n};\Omega)&=\int_\Omega z\cdot D\chi_{E_{t_n}}= -\int_{E_{t_n}} \div z \, dx \\
\to& -\int_{E_t} \div z \, dx= \int_\Omega z\cdot D\chi_{E_t} \leq \Per(E_t;\Omega) ,\quad \text{as } n\to \infty.
\end{align*}
Therefore, \eqref{um-super level}  holds for $t$. 
We thus deduce that the assertion holds for any $t\in [0,1)$.
\end{proof}

\begin{remark}
The existence of a minimizer of \eqref{Mini pb: finite perimeter} can be proved by using the  classical method of Calculus of Variation for every $t\in [0,1)$. See \cite{HawkinsShaoChen2001} for a related problem.
\end{remark}

\medskip
\begin{lem}\label{Lem: ordering}
Let $t'<t$.
If    $F_t$ and $F_{t'}$ are minimizers of \eqref{Mini pb: finite perimeter} with $t$ and $t'$, respectively, then $F_t \subseteq F_{t'}$.
\end{lem}
\begin{proof}
We clearly have
\begin{align*}
  \gamma \Per(F_t;\Omega)  +   \int_{F_t }     p  t^{p-1} \left( P_h -\rho_s U^{\mathrm{vdW}}   \right)   \, dx   
 \leq   \gamma \Per(F_t\cap U_{t'};\Omega)  +   \int_{F_t  \cap F_{t'}}     p t^{p-1} \left( P_h -\rho_s U^{\mathrm{vdW}}   \right)   \, dx
\end{align*}
and
\begin{align*}
  \gamma \Per(F_{t'};\Omega)  +  \int_{F_{t'} }      p ({t'})^{p-1} \left( P_h -\rho_s U^{\mathrm{vdW}}   \right)   \, dx  
 \leq   \gamma \Per(F_t\cup F_{t'};\Omega)  +\int_{F_t\cup F_{t'}  }     p ({t'})^{p-1} \left( P_h -\rho_s U^{\mathrm{vdW}}   \right)  \, dx.
\end{align*}
Because
$$ 
\Per(F_t\cap F_{t'};\Omega)+\Per(F_t\cup F_{t'};\Omega) \leq \Per(F_t;\Omega) + \Per(F_{t'};\Omega),
$$
we deduce that
\begin{align*}
    & ({t'})^{p-1}\left[ \int_{ F_{t'}}  \left(   P_h -\rho_s U^{\mathrm{vdW}}     \right) \, dx  -\int_{F_t\cup F_{t'}}    \left(   P_h -\rho_s U^{\mathrm{vdW}}   \right) \, dx \right]  \\
\leq  &   t^{p-1}  \left[ \int_{F_t\cap F_{t'}  } \left(   P_h -\rho_s U^{\mathrm{vdW}}     \right) \, dx -  \int_{F_t}    \left(   P_h -\rho_s U^{\mathrm{vdW}}    \right) \, dx \right] ,
\end{align*}
i.e.
\begin{align*}
    ({t'})^{p-1} \int_{F_t \setminus F_{t'}}  \left(   P_h -\rho_s U^{\mathrm{vdW}}     \right)\, dx   
\geq   t^{p-1}  \int_{F_t \setminus F_{t'}}    \left(   P_h -\rho_s U^{\mathrm{vdW}}     \right)\, dx . 
\end{align*}
But $t'<t$. This implies that $ F_t \subseteq F_{t'}$.
\end{proof}

\begin{prop}
For all but countably many $t\in (0,1]$, the minimizer of \eqref{Mini pb: finite perimeter}  is unique, i.e. $E_t$.
\end{prop}
\begin{proof}
Fix $t \in (0,1)$ and assume that $F$ is a minimizer   of \eqref{Mini pb: finite perimeter}.
Take an  arbitrary increasing sequence $\{s_n\}_{n=1}^\infty \subseteq (0,1)$ and an arbitrary decreasing sequence $ \{t_n\}_{n=1}^\infty \subseteq (0,1)$ such that $\lim\limits_{n\to\infty} s_n =t = \lim\limits_{n\to\infty} t_n$.

It follows from Proposition~\ref{Prop: nonpolar to finite perimeter} and Lemma~\ref{Lem: ordering} that
$$
\bigcup\limits_n E_{t_n} \subseteq F \subseteq \bigcap\limits_n E_{s_n} .
$$
Note that 
$$
\bigcap\limits_n E_{s_n} =  E_t\cup \{u=t\}\quad \text{and}\quad \bigcup\limits_n E_{t_n}=E_t.
$$
However, there are only countably many  $t$ such that $\cL^3( \{u=t\})>0$. This implies that
$$
F = E_t \quad \text{for a.a. } t\in [0,1).
$$
This completes the proof.
\end{proof}

%%%%%%%%%%%%%%%%%%%%%%%%%%%%%%%%%%%%%%%%%%%%%%%%%%%%%%%%%%%%

\begin{prop}\label{Prop: regularit of mini surface}
For any $t\in [0,1)$, the singular set of $E_t$ is contained in $ \Sigma_0 \cup \Sigma_1 $ and $\partial   E_t \setminus (\Sigma_0 \cup \Sigma_1) $ is of class $C^\infty$.
\end{prop}
\begin{proof}
%It is well-known that the singular set of $E_t\setminus (\partial \Omega_0 \cup \partial \Omega_1)$ is empty, see \cite{Ambrosio, MR775682} for example.
For any $x\in \partial   E_t \cap \Omega_t $, for sufficiently small $r>0$, the ball $B(x,r)$ is contained in $\Omega_t$. For any local perturbation of $E_t$ in $B(x,r)$, i.e. a set $F$ of finite perimeter such that $F \Delta E_t =(F\setminus E_t) \cup (E_t\setminus F) \subset\!\subset B(x,r)$, we have
\begin{align*}
\Per(E_t; B(x,r)) \leq & \Per(F;B(x,r)) + C \int_{B(x,r)} p t^{p-1} \left(P_h  - \rho_s U^{\mathrm{vdW}} \right)\, dx \\
\leq & \Per(F;B(x,r)) + C r^{2+\delta} 
\end{align*}
by H\"older inequality for any $\delta\in (0,1)$. Note that the constant $C$ in the above inequality is independent of the position of $x$. 
Hence $E_t\cap \Omega_t$ is almost minimal in $\Omega_t$ in the sense of \cite[Definition~1.5]{Tam94}. Therefore, \cite[Theorem~1.9]{Tam94} implies that
the singular set of $E_t $ is contained in $\Sigma_0\cup \Sigma_1$ and $\partial   E_t \setminus (\Sigma_0 \cup \Sigma_1) $ is a $C^1$-hypersurface. Then the assertion follows from the standard regularity theorem of non-parametric minimizing surfaces, see \cite{MR775682} for example.
%%%%%%%%%%%%%%%%%%%%%%%
%regularity of minimal surface
%%%%%%%%%%%%%%%%%%%%%%%%%
For the reader's convenience, we will  state a proof here.
For every $x_0\in \partial E_t\setminus (\Sigma_0 \cup \Sigma_1)$,  
denote by $\Hz$ the tangent plane of $\partial E_t$ at $x_0$. 
Near $x_0$, we can rewrite the coordinates in the form $x=(y,z)$, where $y$ is the coordinates in $H$ and $z$ is the coordinate in the normal direction of $H$. We use the convention $z=y=0$ at $x_0$.
For sufficiently small $r>0$, let $U_r=B(x,r)\cap \Hz$. 
Build a  cylinder $C_r=U_r \times (-r,r)\subset\subset \Omega_t$ in $(y,z)-$coordinates centered at $x_0$.
Inside $C_r$, we can express $\partial E_t$  as the graph of a $C^1-$function $v$:
$$
z=v(y),\quad y\in U_r.
$$
See Figure~\ref{tangent}.
\begin{figure}
\centering
\includegraphics[width=0.5\columnwidth]{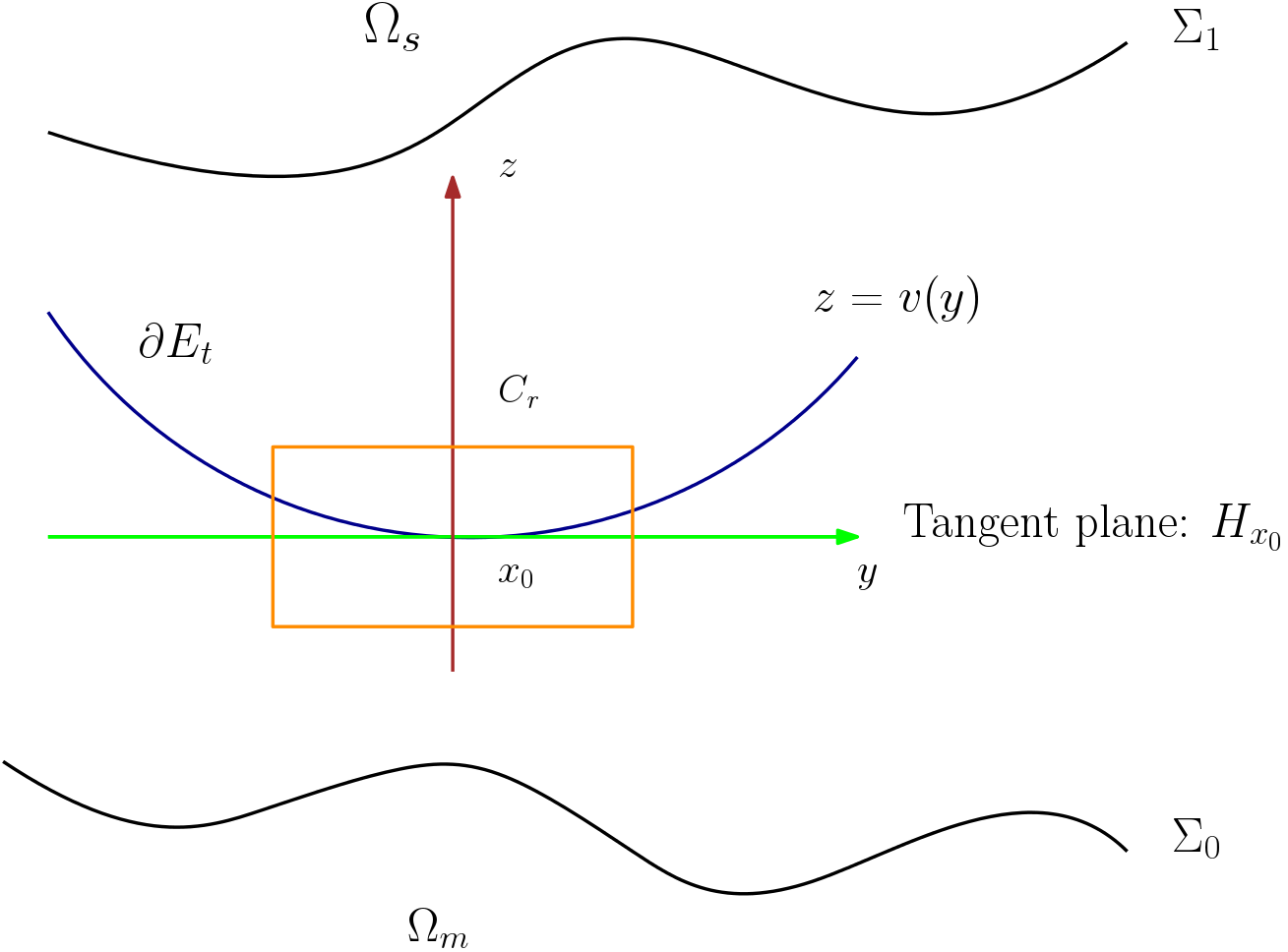}
   \caption{A coordinate system near $x_0\in \partial E_t\setminus (\Sigma_0 \cup \Sigma_1)$}
\label{tangent}
\end{figure}
Then 
\begin{align*}
& \gamma \Per(E_t;C_r)  +  \int_{C_r \cap E_t} p t^{p-1} \left( P_h -  \rho_s  U^{\mathrm{vdW}} \right)\, dx  \\
=& \gamma \int_{U_r } \sqrt{1+ |\nabla_y v(y)|^2}\, dy  + \int_{ U_r } \int_0^{v(y)} \left( P_h -  \rho_s  U^{\mathrm{vdW}}(y,z) \right)\, dzdy.
\end{align*}
By the fundamental theorem of Calculus, $v$ solves
%$$
%A(v )v =f(y,v(y))  
%$$
%in $U_r $, where
%$$
%A(v)w= \frac{\Delta_y w (y)}{\sqrt{1+ |\nabla_y v (y)|^2 }}    -   \frac{(\nabla_y v)  \nabla^2_y w(\nabla_y v)^T}{(\sqrt{1+ |\nabla_y v (y)|^2 })^3}  
%$$
%and
%$$
% f(y,z) =\left(P_h -  \rho_s  U^{\mathrm{vdW}}(y,z)  \right)/\gamma.
%$$ 
%Note that  $f   \in C^\infty(\overline{C}_r)$.
%Hence $v$ solves
\begin{equation*}
\left\{\begin{aligned}
  A(v)v &= f(y,v(y) )  &&\text{in}&& U_r;\\
v &=g &&\text{on}&&\partial U_r
\end{aligned}\right.
\end{equation*}
for some $g\in C^1(\partial U_r)$. Here
$$
A(v)w= \frac{\Delta_y w (y)}{\sqrt{1+ |\nabla_y v (y)|^2 }}    -   \frac{(\nabla_y v)  \nabla^2_y w(\nabla_y v)^T}{(\sqrt{1+ |\nabla_y v (y)|^2 })^3}  ,\quad  f(y,z) =\left(P_h -  \rho_s  U^{\mathrm{vdW}}(y,z)  \right)/\gamma.
$$
By choosing $r>0$ sufficiently small, one can infer from \cite[Theorems~16.10]{MR737190} that $v\in C^2(\overline{U_r})$.
The remaining regularity follows from  a bootstrapping argument, cf. \cite[Theorems~6.13 and 6.17]{MR737190}.
\end{proof}

\medskip
\begin{remark}
If we assume, in addition, that  $\Sigma_i \in C^{1,1}$ for $i=0,1$, then 
following the argument in \cite[Section~1.14(iv)]{Tam94}, one can show that the singular set of $E_t$ is empty and $\partial E_t \in C^{1,1}$.
Since this fact will not be used   below, to keep the article in a reasonable length, we will not provide a rigorous proof here.
\end{remark}

\medskip

\begin{prop}
\label{Prop: regularity of minimizer}
The jump set, $J_{\um}$, of  $\um$ is contained in $\Sigma_0 \cup \Sigma_1$.
% up to a set of zero $\cH^2-$measure.
\end{prop}
\begin{proof}
The proof follows the idea in \cite[Theorem~3.4]{MR2368971}. By \eqref{jump set}, it suffices to show that for any $t_1<t_2\in [0,1)$ and $t_1,t_2\in \mathbb{Q}$, it holds
$$
(\partial E_{t_1} \cap \partial E_{t_2})\setminus (\Sigma_0 \cup \Sigma_1)=\emptyset. 
$$
Assume that $x_0\in (\partial E_{t_1} \cap \partial E_{t_2})\setminus (\Sigma_0 \cup \Sigma_1)$.
By Proposition~\ref{Prop: regularit of mini surface}, both $\partial E_{t_1}$ and $\partial E_{t_2}$  are regular in a neighbourhood of $x_0$.
From the fact $E_{t_2} \subseteq E_{t_1}$, we deduce that the tangent space of $E_{t_2}$ and $ E_{t_1}$ at $x_0$ agree.
Denote  the tangent space by $\Hz$. We define the coordinates in the form $x=(y,z)$ and the   cylinder $C_r=(-r,r)\times U_r$ as in the previous proof.  
Then we can express $E_{t_i}$ with $i=1,2$  as  graphs over $U_r$ as
$$
z=v_i(y) \quad i=1,2 
$$
with $v_i\in C^\infty(U_r)$.
$E_{t_2} \subseteq E_{t_1}$ implies that $v_1\geq v_2$ in $U_r$.
%Locally, the minimizer of \eqref{Mini pb: finite perimeter} with $t=t_i$ can be written as
Similar to the previous proof, we have
$$
 \gamma {\div}_y \left( \frac{\nabla_y v_i(y)}{\sqrt{1+ |\nabla_y v_i(y)|^2 }}  \right) = p t_i^{p-1} \left( P_h   -  \rho_s U^{\mathrm{vdW}}(y,v_i(y)) \right) .
$$
Since $t_2>t_1$, $v_i(0)=0$, $\nabla_y v_i(0)=0$,  by choosing $r>0$ small enough, we have
\begin{align*}
  p t_2^{p-1} \left( P_h   -  \rho_s U^{\mathrm{vdW}}(y,v_2(y)) \right) \Big(\sqrt{1+ |\nabla_y v_2(y)|^2} \Big)^3 
 >   p t_1^{p-1} \left( P_h   -  \rho_s U^{\mathrm{vdW}}(y,v_1(y)) \right) \Big(\sqrt{1+ |\nabla_y v_1(y)|^2} \Big)^3
\end{align*}
for all $y\in U_r$. This implies that
\begin{align*}
  \left(1+ |\nabla_y v_2 |^2 \right)\Delta_y v_2  - \nabla_y v_2  \nabla^2_y v_2  \nabla_y v_2  
>   \left(1+ |\nabla_y v_1 |^2 \right)\Delta_y v_1  - \nabla_y v_1  \nabla^2_y v_1  \nabla_y v_1 
\end{align*}
in $U_r$.
In view of the boundary condition $v_1\geq v_2$ on $\partial U_r$, we infer from \cite[Theorem~10.1]{MR737190} that $v_2<v_1$ in $U_r$, which contradicts $v_1(x_0)=v_2(x_0)$. Therefore, $(\partial E_{t_1} \cap \partial E_{t_2} )\setminus (\Sigma_0\cup \Sigma_1)=\emptyset$.
\end{proof}

\begin{remark}\label{Rmk: continuity of u}
In particular, Proposition \ref{Prop: regularity of minimizer} implies that   $u\in C(\Omega_t)$. 
\end{remark}

%%%%%%%%%%%%%%%%%%%%%%%%%%%%%%%%%%%%%%%%%%%%%%%%%%%%%%%%%%%%

%%%%%%%%%%%%%%%%%%%%%%%%%%%%%%%%%%%%%%%%%%%%%%%%%%%%%%%%%%%%

\subsection{Necessary Conditions for the Formation of a Sharp Interface}\label{Section 5.3}

In this section, we first consider the case that $\Omega_t$ is connected. 
In order to state the main theorem of this section, we define the orientations of $\Sigma_i$ in such a way that
\begin{itemize}
\item the outer normal of $\Sigma_1$ points into $\Omega_t$, and
\item the outer normal of $\Sigma_0$ points into $\Omega_s$.
\end{itemize}
With these conventions, a sphere of radius $R>0$ has constant mean curvature $-1/R$.

\medskip
\begin{theorem}\label{Thm: sharp interface}
Suppose that $\Omega_t$ is connected and $\Sigma_i$,  for $i=0,1$, are  $C^2-$closed  surfaces. Let $\kappa$ be the mean curvature of $\Sigma_1$. 
If $\kappa(\p)>0$ for some $\p\in \Sigma_1$, then there is no sharp solute-solvent interface, that is, the minimizer $\um$ of \eqref{non penalized total energy functional final form} is not the characteristic function of a set $E$ of finite perimeter with $\Omega_m \subseteq E \subseteq \Omega_w$. 
\end{theorem}
\begin{proof}
Assume, to the contrary, that there exists a   set $E$ of finite perimeter such that $\Omega_m \subseteq E \subseteq \Omega_w$ and $\chi_E$ minimizes \eqref{non penalized total energy functional final form}.

(1) By the De Giorgi Theorem, cf. \cite[Theorem~3.59 and Example 3.68]{MR1857292}, we have
$$
\partial^* E \subseteq J_{\chi_E}\subseteq \Sigma_0 \cup \Sigma_1.
$$
For every $x\in \Omega_t\cap E$, \eqref{red bdry} implies that 
$ 
\Per(E;B(x,r))   =0
$ 
for all  $r>0$  so small that $B(x,r)\subset \Omega_t$. Thus the isoperimeteric inequality, cf. \cite[Theorem~5.6.2]{MR1158660}, implies that
$$
\min \{ \cL^3(B(x,r)\cap E), \cL^3(B(x,r)\setminus E)  \}^{2/3}\leq C \Per(E;B(x,r))=0.
$$
%This shows that either $\cL^3(B(x,r)\cap E)=0$ or $\cL^3(B(x,r)\setminus E)=0$. 
%We claim the latter holds for all $x\in \Omega_t$ when $\cL^3(E \cap \Omega_t)>0$..
%Assume, to the contrary, that there exist two distinct points $x_1,x_2\in \Omega_t$ such that $\cL^3(B(x_1,r)\cap E)=0$ and $\cL^3(B(x_2,r)\setminus E)=0$. 
If $\cL^3(E \cap \Omega_t)>0$, assume that there exist two distinct points $x_1,x_2\in \Omega_t$ such that $\cL^3(B(x_1,r)\cap E)=0$ and $\cL^3(B(x_2,r)\setminus E)=0$.
Since $\Omega_t$ is connected, we can find a continuous path $\gamma:[0,1]\to \Omega_t$ such that
$$
\gamma(0)=x_1 ,\quad \gamma(1)=x_2.
$$
Further assume that $r>0$ is so small that $B(x,r)\subset \Omega_t$ for all $x\in \gamma([0,1]).$ Then for any $x\in \gamma([0,1]) \cap B(x_1,r)$, we have  $\cL^3(B(x ,r)\cap E)=0$. Repeating this argument for finitely many times shows that $\cL^3(B(x_2,r)\cap E)=0$. A contradiction. Therefore, $\cL^3(B(x,r)\setminus E)=0$ for all $x\in \Omega_t$ and all $r>0$ so small that $B(x,r)\subset \Omega_t$. 
We immediately infer that
$$
\cL^3(\Omega_t \setminus E)=0 
$$
and thus $\chi_E =\chi_{\Omega_w}$ a.e. To sum up, we have either  $E=\Omega_m$ or $E=\Omega_w$.

(2) Consider the case that $E= \Omega_m $, or equivalently $\um=\chi_E$. Define $E_t$ as in \eqref{super-level}. Then for each $t\in [0,1)$, $E_t=\Omega_m$. Therefore, $\chi_{\Omega_m}$ is the unique minimizer of \eqref{Mini pb: finite perimeter} for every $t\in [0,1)$.

Since $\Sigma_1$ is $C^2$,  
it has a tubular neighborhood $B_\a(\Sigma_1)$ of width $\a>0$,   cf. \cite[Exercise 2.11]{MR737190} and \cite[Remark 3.1]{MR4160135}. 
%it satisfies the uniform ball condition of radius $\a$ for some $\a>0$.  See for instance \cite[Exercise 2.11]{MR737190}.
Given any $\rho \in C^1(\Sigma_1)$ with $0\leq \rho \leq 1$, the map
$$
\Psi_\rho: (-\a,\a) \times\Sigma_1 \to \R^3 : \,  (\varepsilon,\p)\mapsto \p+\varepsilon \rho(\p) \ns(\p),
$$ 
is a $C^1$-diffeomorphism  onto its image, where $\ns$ is the outward unit normal of $\Sigma_1$ pointing into $\Omega_t$. 
%In this article, $\p$ always denotes a generic point on the reference hypersurface $\Sigma$.
Put $\Gamma_{\varepsilon}:= \Psi_\rho(\varepsilon , \Sigma )$ and $\Omega_{\varepsilon}$ as the region enclosed by $\Gamma_{\varepsilon}$. 
Observe that  $\Omega_0=\Omega_m$ and 
$$
\Omega_m\subseteq \Omega_{\varepsilon } \subseteq \Omega_w
$$ 
for all $\varepsilon \in [0, \a)$ with sufficiently small $\a$. 
Define a functional 
$$
F_t(\varepsilon)=\gamma \Per(\Gamma_{\varepsilon};\Omega)  +   \int_{\Omega_\varepsilon } pt ^{p-1 }  \left( P_h-  \rho_s U^{\mathrm{vdW}}     \right) \, dx,\quad \varepsilon\in [0,\a).
$$
Note that  $F_t(\varepsilon)\geq F_t(0)$. By \cite[Equation~(21)]{HawkinsShaoChen2001},
\begin{align*}
 \lim\limits_{\varepsilon\to 0^+} \frac{F_t(\varepsilon)-F_t(0)}{\varepsilon} 
=& \int_{\Sigma_1} \rho \left( -2\gamma \kappa + pt ^{p-1 } P_h- pt ^{p-1 }\rho_s U^{\mathrm{vdW}} \right)\, d\Sigma_1  ,
\end{align*}
where $d\Sigma_1$ is the volume element on $\Sigma_1$. 
Thus 
$$
\int_{\Sigma_1} \rho \left( -2\gamma \kappa + pt ^{p-1 } P_h- pt ^{p-1 }\rho_s U^{\mathrm{vdW}} \right)\, d\Sigma_1\geq 0
$$
for all $\rho \in C^1(\Sigma_1)$ with $\rho\geq 0$. This implies that 
$$
  pt ^{p-1 } P_h- pt ^{p-1 }\rho_s U^{\mathrm{vdW}} \geq 2\gamma \kappa
$$
for all $t\in [0,1)$. Taking $t=0$ above yields
$$
0\geq \kappa \quad \text{on } \Sigma_1.
$$
This is a necessary condition for $E= \Omega_m $. Therefore, if $\kappa(\p)>0$ for some $\p\in \Sigma_1$, then $E \neq  \Omega_m $.

(3)
Let $\widehat{\kappa}$ be the mean curvature of $\Sigma_0$.
If $E=  \Omega_w $, then  following the above argument, we conclude that
$$
\int_{\Sigma_0} \rho \left( -2\gamma \widehat{\kappa} + pt ^{p-1 } P_h- pt ^{p-1 }\rho_s U^{\mathrm{vdW}} \right)\, d\Sigma_0\geq 0
$$
for all $\rho \in C^1(\Sigma_0)$ with $\rho\leq 0$ and $t\in [0,1)$. Here $d\Sigma_0$ is the volume element on $\Sigma_0$. 
Pushing $t\to 1^-$ implies that
$$
\widehat{\kappa} \geq \frac{p  P_h- p \rho_s U^{\mathrm{vdW}}}{2\gamma}>0 
$$
is a necessary condition for $E=\Omega_w$
However, it is well known that there is no closed hypersurface with everywhere positive mean curvature in $\R^3$.
Therefore, $E\neq \Omega_w$
\end{proof}

\begin{remark}\label{Rmk: stable}
The mean curvature condition $\kappa(\p)>0$ for some $\p\in \Sigma_1$ is satisfied by almost all macromolecules. 
This explains why diffuse interfaces are indeed more realistic in real-world solvation processes. 
It is equally important  to point out that the mean curvature condition is in some sense ``stable". 
Recall that the Hausdorff metric on compact subsets $K\subset\R^n$, $n\in \N$, is defined by
$$
d_{\cH}(K_1,K_2)=\max\left\{ \sup\limits_{x\in K_1}d(x,K_2), \sup\limits_{x\in K_2}d(x,K_1) \right\}.
$$
Given a closed surface $\Sigma$ in $\R^3$, its second normal bundle is given by
$$
\cN^2 \Sigma=\{(\p, \nu_{\Sigma}(\p), \nabla_{\Sigma} \nu_{\Sigma}(\p) ) : \, \p\in \Sigma \}\subset \R^3\times \R^3 \times \R^9,
$$
where $\nabla_{\Sigma}$ is the surface gradient defined by 
$$
\nabla_{\Sigma} \vec{v}(\p)= P_{\Sigma}(\p) \nabla \vec{v}(\p),\quad \vec{v}\in C^1(B_r(\Sigma),\R^3)
$$
for some $r>0$.
Here $P_{\Sigma}(\p)=I-\nu_{\Sigma}(\p)\otimes \nu_{\Sigma}(\p)$. 
Denote by $\cM$ the set of all connected closed surfaces in $\R^3$. Equipped with the metric $d_{\cH}$, $\cM$ is a Banach manifold, cf. \cite{MR3297714, MR3524106}. If a connected component, $\M_1$, of $\Sigma_1$ satisfies the condition in Theorem~\ref{Thm: sharp interface}, then any $\Sigma\in \cM$ that is sufficiently close to $\M_1$ with respect to the metric $d_{\cH}$ satisfies the same condition.
\end{remark}

\begin{remark}\label{Rmk: cavity}
The connectedness condition of $\Omega_t$ was used in the   proof of Theorem~\ref{Thm: sharp interface}.
It is well-known that cavities may appear inside macromolecules, which corresponds to the situation of disconnected $\Omega_m$.  
In the case of $N$ cavities inside $\Omega_m$, $\Omega_t$ consists of $N+1$ connected components. More precisely,
$$
\Sigma_1=\bigcup\limits_{j=0}^N \Gamma_j,
$$ 
where $\Gamma_j$ are $C^2-$closed and connected hypersurfaces and  $\Gamma_j$, $j=1,\cdots,N$, is the boundary of the $j-$th cavity.
Correspondingly, 
$$\Omega_t=\bigcup\limits_{j=0}^N \Omega_{t,j},
$$
where $\Omega_{t,j}$ are the connected components of $\Omega_t$ and $\Omega_{t,j}$, $j=1,\cdots,N$, is the $j-$th cavity bounded by $\Gamma_j$ and $\overline{\Omega}_{t,0}\cap \overline{\Omega}_s\neq \emptyset$.
See Figure~\ref{domain} for a picture illustration of a solute molecule with one cavity inside.
To make the convention of the mean curvature consistent, we define the orientation of $\Gamma_j$ in the following way:
\begin{itemize}
\item the outer normal of $\Gamma_0$ points into $\Omega_{b,0}$;
\item for $j=1,\cdots,N$, the outer normal of $\Gamma_j$ points into $\Omega_m$.
\end{itemize}
Under these conventions, we can follow the proof of Theorem~\ref{Thm: sharp interface} and show that $\Gamma_j$ $(j=1,\cdots,N)$ is a sharp interface iff $\Gamma_j$ has everywhere positive mean curvature, which is impossible.
Therefore, none of the cavities can be purely occupied by the solvent.
\begin{figure}
            \begin{center}
                \begin{tabular}{cc}
                         \includegraphics[width=0.4\columnwidth]{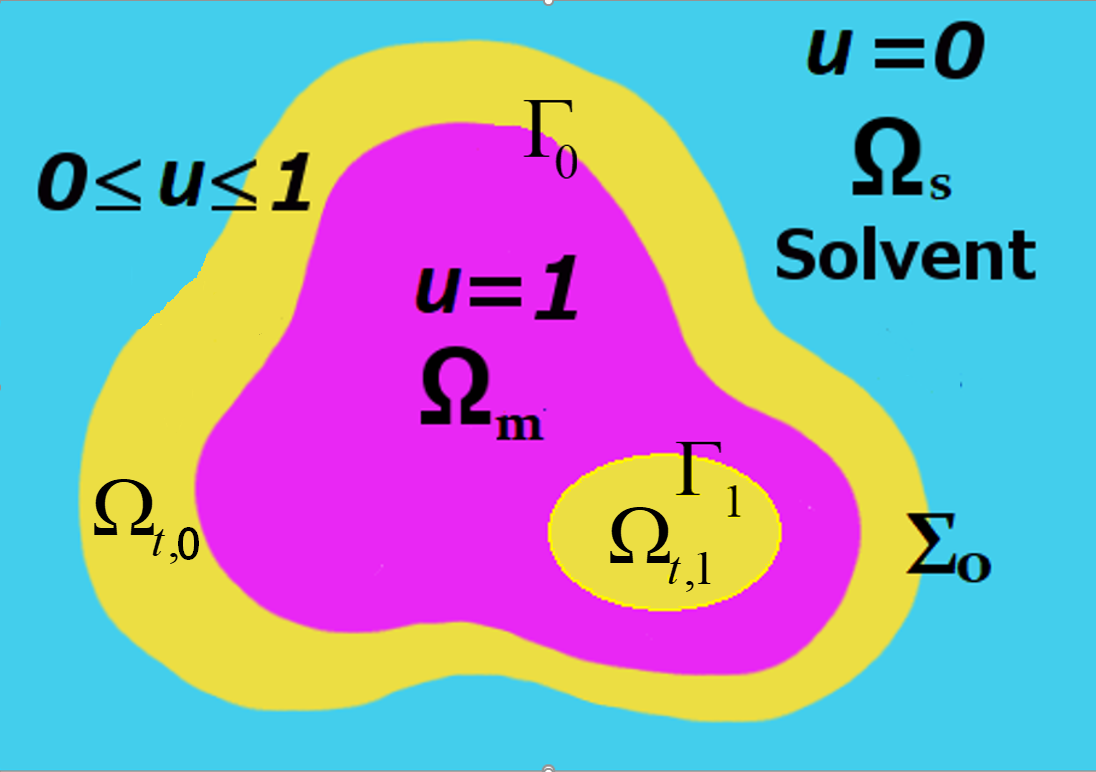}
                \end{tabular}
            \end{center}
             \caption{ Illustration of a solute with one cavity inside.
}\label{domain}
\end{figure}
\end{remark}

%%%%%%%%%%%%%%%%%%%%%%%%%%%%%%%%%%%%%%%%%%%%%%%%%%%%%%%%%%%%

%%%%%%%%%%%%%%%%%%%%%%%%%%%%%%%%%%%%%%%%%%%%%%%%%%%%%%%%%%%%
\section{Numerical Simulations}\label{Section 6}

The non-differentiable structure of \eqref{non penalized total energy functional final form} and the Constraints~\eqref{constrain 1} and \eqref{constrain 2} generate  an essential difficulty in the numerical simulations of \eqref{non penalized total energy functional final form}.
This motives  us to study a sequence of approximation problems.

%%%%%%%%%%%%%%%%%%%%%%%%%%%%%%%%%%%%%%%%%%%%%%%%%%%%%%%%%%%%
\subsection{An Approximation Problem}

Recall the definition of $\{q_k\}_{k=1}^\infty$ from Section~\ref{Section 3}. We introduce a family of perturbed solvation free energy functionals
\begin{align}\label{perturbed energy functional}
\notag
  I_k(u)
=& \gamma \int_\Omega |\nabla u|^{q_k}\, dx + \int_\Omega \left[ P_h u^p  + \rho_s (1-u^p) U^{\mathrm{vdW}}   \right] \, dx  
  \\
&+ \int_\Omega   \left[    \rho_m \psi    - \frac{1}{2}\epsilon(u) |\nabla \psi|^2 - (q_k- u^p)  B(\psi)          \right]\, dx,
\end{align}
where $\psi\in \cA$ satisfies \eqref{GPBEk}. 
We will seek a minimizer of $I_k(\cdot)$ in $\cYk$, c.f. \eqref{Def: cYk}.
For notational brevity, we term the second line of \eqref{perturbed energy functional} $I_{{\rm p},k}(u,\psi)$.
%\begin{align*}
%\cYk=\{ & u\in W^{1,q_k}(\Omega): \,    - \sqrt[p]{q_k}  \leq u \leq \sqrt[p]{q_k}  \text{ a.e. in }\Omega \quad \text{and}\quad u=1\text{ a.e. in }\Omega_m \\
%&\text{and } u=0\text{ a.e. in }\Omega_s\} .
%\end{align*}

Let $\um$ be a minimizer of \eqref{non penalized total energy functional final form}  in $\cY$  and $\pmi=\psi_{\um}$ be the solution of \eqref{GPBE} with $u=\um$.

To prepare for the main result of this section, we introduce 
$$
\Omega_{j,k}:=\{ x  \in \Omega:   {\rm dis}( x , \Omega_j)<1/k \}, \quad j\in \{m,s\}
$$ 
and
$$
\cX_k:=\{u\in \cY : u\equiv 1 \,\, \text{ in } \Omega_{m,k} \quad \text{and} \quad  u\equiv 0\,\, \text{ in } \Omega_{s,k} \},\quad k\in\N,
$$
and quote the following two lemmas from \cite{Shao2022constrained}.

\medskip
\begin{lem}{\em (\cite[Lemma~2.6]{Shao2022constrained})}
\label{tech lem 1}
For every $f\in \cX_k$, there exists a sequence $\{f_n\}_{n=1}^\infty \subset   C^\infty( \overline{\Omega})$ satisfying Constraints~\eqref{constrain 1} and \eqref{constrain 2} such that
\begin{itemize}
\item[{\em (i)}] $f_n\to f$ in $L^1(\Omega)$, and
\item[{\em (ii)}] $\|D f_n\|(\Omega) \to \|D f \|(\Omega)$ as $n\to \infty$.
\end{itemize}	
\end{lem}

\medskip
\begin{lem}{\em (\cite[Lemma~2.7]{Shao2022constrained})}
\label{tech lem 2}
For every $f\in \cY$, we define $\{f_k\}_{k=1}^\infty \subset BV(\Omega)$ by
\begin{align*}
f_k(x)=
\begin{cases}
1 ,\quad & x\in \Omega_{m,k}\\
0, & x\in \Omega_{s,k}\\
f(x), & \text{elsewhere}.
\end{cases}
\end{align*}
If $\Sigma_i\in C^2$ with $i\in \{0,1\}$, then
\begin{itemize}
\item[{\em (i)}] $f_k\to f$ in $L^1(\Omega)$ and
\item[{\em (ii)}] $\|D f_k\|(\Omega) \to \|D f \|(\Omega)$ as $k\to \infty$.
\end{itemize}	
\end{lem}

\medskip
The theoretic basis of the numerical simulation is  the following theorem.

\begin{theorem}\label{Thm: existence of minimizer Ek}
%Given $\psi_\infty\in W^1_\infty(\Omega)$,
For each $k=1,2,\cdots$, there exists a unique $\umk \in \cY_k\cap \cY$ such that $I_k(\umk)=\underset{{u\in \cYk}}{\min} I_k(u)$.
If, in addition,   $\Sigma_i\in C^2$, $i\in \{0,1\}$,
$$
\lim\limits_{k\to \infty} I_k(\umk)=I(\um),
$$ 
and  as $k\to \infty$  
$$
\umk \to \um \quad \text{in } L^r(\Omega)
$$
for all $r\in [1,\infty)$  and
$$
\pmk \to \pmi \quad \text{in } H^1(\Omega),
$$
where $\pmk=\psi_{\umk}$ is the solution to \eqref{GPBEk} with $u=\umk$.
\end{theorem}
\begin{proof}
(i) The existence and uniqueness of a minimizer $\umk \in \cYk$ of $I_k(\cdot)$ for each $k$ can be proved in the same manner as in Theorem~\ref{Thm: existence of global minimizer}.

(ii) We will show that $\umk$ is a global minimizer of $I_k(\cdot)$ iff $(\umk,\pmk)$ is a saddle point of
\begin{align}\label{def: Lk}
\hspace*{-2em}
    L_k (u,\psi)   
  := \! \int_\Omega \left[ \gamma |\nabla u|^{q_k}   + P_h u^p  + \rho_s (1-u^p) U^{\mathrm{vdW}} \right] \, dx    
 + \int_\Omega   \left[     \rho_m \psi    - \frac{1}{2} \epsilon (u) |\nabla \psi|^2 - (q_k- u^p  ) B(\psi)  \right]\, dx 
\end{align}
in $\cYk\times \sA$, where
$$
\sA:=\{v\in \cA: \|v\|_{H^1}\leq  \widetilde{C}_0 \text{ and }  \|v\|_\infty \leq  \widetilde{C}_0\}.
$$ 
Here $\widetilde{C}_0 $ is the constant in \eqref{est for psi 1}.
Proposition~\ref{Prop: GPBE}   shows that   $\pmk\in \sA$.
Denote by $\cS_k$ the set of all saddle points of $L_k$. Recall that $(u_0,\psi_0)\in \cS_k$ iff 
\begin{equation}
\label{def: saddle pt}
L_k(u_0, \psi)\leq L_k(u_0,\psi_0)\leq    L_k(u,\psi_0), \quad \forall (u,\psi)\in \cYk \times \sA .
\end{equation}
%Define $M_k:=I_k(\umk)$.
It follows from Proposition~\ref{Prop: GPBE} and Theorem~\ref{Thm: existence of global minimizer} that
$$
I_k(\umk)=:M_k=\underset{{u\in \cYk} }{\min}\underset{{\psi\in \sA}}{\max} L_k(u,\psi) .
$$
%and $u\in \cZ_q$ iff $(u,\psi_u)$ is a solution of the above minimax problem.

Note that   $\cYk$ and $\sA$ are closed and convex in $W^{1,q_k}(\Omega)$ and $H^1(\Omega)$, respectively. Moreover, 
$$
[u\mapsto L_k(u,\psi)] \text{ is strictly convex and lower semi-continuous  } \forall \psi\in \sA,
$$ 
and
$$
[\psi\mapsto L_k(u,\psi)] \text{ is strictly concave and upper semi-continuous  } \forall u\in \cYk.
$$
Since $\sA$ is bounded in $H^1(\Omega)$,
\cite[Remark~VI.2.3]{MR1727362} implies that 
$$
\underset{{\psi\in \sA}}{\max}  \inf\limits_{u\in \cYk}  L_k(u,\psi) =\underset{{u\in \cYk} }{\min}\underset{{\psi\in \sA}}{\max}    L_k(u,\psi)=M_k.
$$
It follows from  the direct method of Calculus of Variation that the infimum is achieved. Therefore,
\begin{equation}
\label{saddle point} 
\underset{{\psi\in \sA}}{\max}  \underset{{u\in \cYk}}{\min}  L_k(u,\psi) =\underset{{u\in \cYk} }{\min}\underset{{\psi\in \sA}}{\max}    L_k(u,\psi)= L_k(\umk,\pmk).
\end{equation}
By \cite[Proposition~VI.1.2]{MR1727362}, $(\umk,\pmk)\in \cS_k$.
Conversely, 
if $(u_0,\psi_0)\in \cS_k$, then \eqref{def: saddle pt} and Proposition~\ref{Prop: GPBE} show that $\psi_0 $ is the solution of \eqref{GPBEk} with $u=u_0$. What is more, since  \eqref{saddle point} still holds true if we replace $(\umk,\pmk)$ by $(u_0,\psi_0)$, we infer that $u_0=\umk$.
%Therefore, $u_0\in \cZ_q$ iff $(u_0,\psi_{u_0})\in \cS_q$.

If $\cL^3(\{\umk >1 \}\cup \{\umk <0\})>0$, define
\begin{align*}
\bar{u}_{{\rm min},k}(x)=
\begin{cases}
1,\quad & \text{if } \umk(x)>1 ,\\
0,  & \text{if } \umk(x)<0 ,\\
\umk(x),  & \text{elsewhere.}
\end{cases}
\end{align*}
Then direct computations show that
$$
L_k (\bar{u}_{{\rm min},k},\pmk) < L_k (u_k,\pmk).
$$
A contradiction. Hence, $\umk\in \cY$. 
%Let $\psi_k$ be the solution of \eqref{GPBEk} with $u=\umk$. 

(iii) Fix $v\in \cY_k$. Then, by \eqref{bound Eu},
$ 
  G_v^k (\psi_v) < \widetilde{C}_1,
$ 
where $\psi_v$ is the solution to \eqref{GPBEk} with $u=v$.
Then
\begin{align*}
  I_k(v) 
\leq  \gamma \int_\Omega |\nabla v |^{q_k}\, d x +2P_h  {\rm Vol}(\Omega )   - \int_{\Omega\setminus \Omega_s	} \rho_s U^{{\rm vdW}}\, d x  
  +\widetilde{C}_0 \|\rho_m\|_\infty {\rm Vol}(  \Omega_m) 
\leq   C_2 ,
\end{align*}
where $\widetilde{C}_0$ is the constant in Proposition~\ref{Prop: GPBE} and $C_2$  is independent of $k$ and $v$.
This yields that
\begin{align}
\notag C_2 \geq & I_k(\umk) \geq \gamma \int_{\Omega} |\nabla \umk|^{q_k} \, d x +P_h \|\umk\|_p^p +C_3 -\widetilde{C}_1 \\
\label{est q1}
\geq  & \gamma \| \nabla \umk\|_1^{q_k} \left( {\rm Vol}(\Omega) \right)^{1-{q_k}}  +P_h \|\umk\|_p^p+ C_3 -\widetilde{C}_1 ,
%\label{est q2}
%\geq &  \gamma_0^{-{q_k}}   \gamma \|   \umk \|_{3/2}^{q_k} \left( {\rm Vol}(\Omega) \right)^{1-{q_k}}-C_3 -\widetilde{C}_1,
\end{align}
where 
$C_3= \int_{\Omega  \setminus \Omega_m} \rho_s U^{{\rm vdW}}\, d x$. 
%and $\gamma_0$ is the constant in Proposition~\ref{A1: Poincare BV}.
%Here it is understood that a function $w\in \cY$ is automatically extended to be identically zero outside $\Omega$.
We thus infer from \eqref{est q1}   that 
$$
\|  \umk\|_{W^{1,1} }=\|\umk\|_{BV} \leq  C_4
$$
for some $C_4$ independent of $k$. 
Proposition~\ref{A1: compact of BV} implies that there exists a subsequence of $\{\umk\}_{k=1}^\infty$, not relabelled, converging to some $u_0\in \cY$ in $L^1(\Omega)$.  
The Riesz-Thorin interpolation theorem then implies that $\umk \to u_0$ in $L^r(\Omega)$ for all $r\in [1,\infty)$ as $k \to \infty$.
Note that 
$$
  \int_{\Omega} |\nabla \umk |^{q_k} \, d x    
\geq    \| \nabla \umk\|_1^{q_k} \left( {\rm Vol}(\Omega) \right)^{1-{q_k}} .
$$
Then it follows from Propositions~\ref{A1: lsc of BV}  and  \ref{Prop: GPBE converg} that 
$$
I (u_0) \leq \liminf\limits_{k\to \infty}I_{q_k}(\umk).
$$
On the other hand, we define
\begin{align*}
w_n(x)=
\begin{cases}
1 ,\quad & x\in \Omega_{m,n}\\
0, & x\in \Omega_{s,n}\\
u_0(x), & \text{elsewhere}.
\end{cases}
\end{align*}
We will show that
\begin{equation}\label{converg q1}
\limsup\limits_{k\to \infty} I_k (\umk) \leq I(w_n).
\end{equation}
Lemma~\ref{tech lem 1} implies that we can find a sequence $\{w_{n,i}\}_{i=1}^\infty $ such that   $w_{n,i}\in C^\infty(\overline{\Omega})\cap \cY_k$  for all $k$ 
and  
$$
w_{n,i}\to w_n \,\, \text{ in } L^1(\Omega)  \quad \text{and} \quad \|D w_{n,i}\|(\Omega) \to \|D w_n\| (\Omega) \quad \text{as } i\to \infty.
$$
Since $\umk$ minimizes $I_k(\cdot)$ in $\cY_k$, we have
$$
 I_k (\umk) \leq  I_k (w_{n,i}).
$$
Pushing $k\to \infty$, the dominated convergence theorem  and Proposition~\ref{Prop: GPBE converg} imply that
$$
\limsup\limits_{k\to \infty} I_k (\umk) \leq I  (w_{n,i}).
$$
Then Lemma~\ref{tech lem 1} and Proposition~\ref{Prop: GPBE converg} immediately yield  \eqref{converg q1}.
%$$
%\limsup\limits_{n\to \infty} I_{q_n} (u_{q_n}) \leq I_1 (w_k).
%$$
Now Lemma~\ref{tech lem 2} and Proposition~\ref{Prop: GPBE converg} give that
$$
\limsup\limits_{k\to \infty} I_k (\umk) \leq I  (u_0).
$$
Finally, the convergence of $\pmk$ is a direct consequence of Proposition~\ref{Prop: GPBE converg}.

(iv) Denote by $\psi_k$ the solution of \eqref{GPBEk} with $u=\um$. Then by Proposition~\ref{Prop: GPBE},
\begin{align*}
I(\um) \geq  \Inp(\um) + \Ip(\um,\psi_k) \geq \Inp(\um) + I_{{\rm  p},k}(\um,\psi_k) 
\geq   I_k(\umk).
\end{align*}
This yields
$$
I(\um) \geq \lim\limits_{k\to \infty}I_k(\umk)=I(u_0).
$$
By the uniqueness of a global minimizer of $I(\cdot)$, we conclude that $u_0=\um$.
\end{proof}

%%%%%%%%%%%%%%%%%%%%%%%%%%%%%%%%%%%%%%%%%%%%%%%%%%%%%%%%%%%%
\subsection{Variation of Solvation Free Energy}

%%%%%%%%%%%%%%%%%%%%%%%%%%%%%%%%%%%%%%%%%%%%%%%%%%%%%%%%%%%%
%saddle points
%%%%%%%%%%%%%%%%%%%%%%%%%%%%%%%%%%%%%%%%%%%%%%%%%%%%%%%%%%%%

Motivated by Theorem~\ref{Thm: existence of minimizer Ek}, we will study the numerical simulations of the approximating functional  \eqref{perturbed energy functional}.
As the first step, we will  derive the variational formulas  of \eqref{perturbed energy functional} at $\umk$. 
Recall that $\umk$ minimizes \eqref{perturbed energy functional} in $\cYk$ iff $(\umk,\pmk) $ is a saddle point of \eqref{def: Lk}  in $\cYk \times \sA$, where $\pmk$ solves \eqref{GPBEk} with $u=\umk$. 
This means that
$$
L_k(\umk,\pmk)=\underset{{u\in \cYk}}{\min} L_k(u,\pmk).
$$
Given any $\phi\in C^\infty_0(\Omega_t)$, as $\umk\in \cY$, for sufficiently small $\delta>0$,
$$
\umk+\varepsilon \phi \in \cYk,\quad |\varepsilon|<\delta.
$$
Therefore,  we can verify that $\umk$ satisfies
\begin{align*}
%\label{EL eq}
&\gamma \int_\Omega  q_k |\nabla \umk |^{q_k-2} \nabla \umk   \cdot \nabla \phi\, dx +\int_\Omega \left[ p\umk^{p-1} \left( P_h  - \rho_s   U^{\mathrm{vdW}}   \right) \phi \right] \, dx  \\
+ & \int_\Omega \left[ p\umk^{p-1} \left(  B(\pmk)  +\frac{\epsilon_s-\epsilon_m}{2} |\nabla \pmk|^2   \right) \phi \right] \, dx  
  =0
\end{align*}
for all $\phi\in C^\infty_0(\Omega_t)$. 
Therefore, $\umk$ solves
\begin{align*}
%\label{EL eq}
 \gamma q_k \div \left(  |\nabla u |^{q_k-2}\nabla u  \right) -  p u^{p-1} V(\pmk)=0 \quad \text{in }\Omega_t 
\end{align*}
in the weak sense,
where 
$$
V ( \psi)=P_h  - \rho_s   U^{\mathrm{vdW}} +B(\psi)   + \frac{\epsilon_s-\epsilon_m}{2} |\nabla \psi|^2 .
$$
In view of \eqref{GPBEk}, $(\umk,\pmk)$ solves the following elliptic system
\begin{equation}
\label{EL-sys}
\left\{\begin{aligned}
\div (\epsilon(u ) \nabla \psi)  + (q_k-u^p) \sum\limits_{j=1}^{N_c} c_j^\infty q_j e^{- \beta \psi q_j }   &=  -\rho_m   &&\text{in}&&\Omega;\\
\psi&=\psi_\infty   &&\text{on}&&\partial\Omega; \\
\gamma q_k \div \left(  |\nabla u |^{q_k-2}\nabla u  \right) -p u^{p-1} V (\psi)   &=  0   &&\text{in}&&\Omega_t;\\
u & =1 &&\text{on}&&\Sigma_1;\\
u & =0 &&\text{on}&&\Sigma_0. 
\end{aligned}\right.
\end{equation}

\begin{remark}
The approach in this section actually gives a solution to the variational analysis of \eqref{totalfunctional1} with Constraints \ref{constrain 1} and \ref{constrain 2},
which provides a complete answer to a question  in our previous work \cite{Shao2022constrained}.
%This will be discussed in more detail in a subsequent paper.
\end{remark}

%%%%%%%%%%%%%%%%%%%%%%%%%%%%%%%%%%%%%%%%%%%%%%%%%%%%%%%%%%%%
\subsection{Computational methods}\label{Section 6.3}

This section presents the computational methods and algorithms for the solution of the coupled system \eqref{EL-sys} and its associated parameterization process. 
The solution of \eqref{EL-sys} provides a physically sound ``diffuse solute-solvent  interface profile" $u$ and the electrostatic potential $\psi $, and thereby the calculation of the total solvation free energy.

While solving for $u$ and $\psi$, the surface evolution equation and the perturbed PB equation cannot be decoupled and thus need to be solved simultaneously.
In the following, we first describe in more detail  about the solution methods for each equation and their discretized formulations. Then the scheme for the convergence of two coupled equations is presented as well as a simple parameterization approach for optimal parameter values. 

%To solve for $u$ and $\phi$, two involved equations need to be solved simultaneously. The solution of the surface evolution equation requires the solution of  $\phi$. Meanwhile, the solution of the perturbedPB equation requires the knowledge of $u$ and $\epsilon(u)$. In the following, we first describe in more details about the solution methods for each equation and their discretized formulations. Then the scheme for the convergence of two coupled equations is presented as well as a simple parameterization approach for optimal parameter values. 

\subsubsection{The perturbed Poisson-Boltzmann equation}\label{section 6.3.1}

For the solution of perturbed PB (PPB) equation, we adopted the finite difference scheme. Thanks to the continuous dielectric function, an accurate solution can be achieved with a standard second-order center difference scheme. Specifically, for a solvent without salt, the PPB equation can be simplified to a perturbed Poisson equation.  If the position $(x_i,y_j,z_k)$ is represented by the pixel $(i,j,k)$, its discretized form becomes
\begin{eqnarray}\label{disppb}\nonumber
	\epsilon(i+\frac{1}{2},j,k)[\psi(i+1,j,k)-\psi(i,j,k)]&-& \epsilon(i-\frac{1}{2},j,k)[\psi(i-1,j,k)-\psi(i,j,k)]\\ \nonumber
	\nonumber + \epsilon(i,j+\frac{1}{2},k)[\psi(i,j+1,k)-\psi(i,j,k)]&-&
	\epsilon(i,j-\frac{1}{2},k)[\psi(i,j-1,k)-\psi(i,j,k)]\\ \nonumber +
	\epsilon(i,j,k+\frac{1}{2})[\psi(i,j,k+1)-\psi(i,j,k)]&-& \epsilon(i,j,k-\frac{1}{2})[\psi(i,j,k-1)-\psi(i,j,k)] 
	  = -q(i,j,k)/h
\end{eqnarray}
where the uniform grid spacing $h$ is applied at $x$, $y$ and $z$ directions, and $\epsilon(i+\frac{1}{2},j,k)=\epsilon (u(x_i+\frac{1}{2}h,y_j,z_k))$, $ q(i,j,k)$ is used to describe the fractional charge at grid point $(x_i,y_j,z_k)$. The fractional charge is calculated by  the second-order interpolation (trilinear) of the charge density $\rho_m$. Then a standard  linear algebraic equation system  is resulted from the  the discretized perturbed Poisson equation in the form of  $AX=B$, in which $X$ is the targeted solution. Matrix $A$ is the discretization matrix and $B$ is the source term according to the discrete charges. 

The boundary condition of PPB equation is computed via the summation of electrostatic potential contributions of individual atom charges \cite{Geng:2007a}. The resulted linear system can be solved by various linear solvers (like biconjugate gradient in this study) together with pre-conditioners for potential acceleration. $0$ can be used for the initial guess of the solution and convergence tolerance is set as a small number such as $10^{-6}$. It has been shown that the designed PB solver is capable of delivering second-order accuracy \cite{ZhanChen:2010a}.

\subsubsection{The surface evolution equation}\label{section 6.3.2}
The solution of the surface evolution equation can be attained via the following parabolic PDE as done in earlier work \cite{Bates:2009,ZhanChen:2010a}.
\begin{equation}\label{see}
	\frac{\partial u}{\partial t}= |\nabla u |^{2-q_k}\left[\div\left(\gamma q_k\frac{\nabla u}{ |\nabla u |^{2-q_k}}\right)
	+ p u^{p-1} V\right],
\end{equation}
As a result, the steady state solution of Equation \eqref{see} can be directly taken as the solution of the original elliptic equation. 

Computationally, the equation \eqref{see}  can be expanded into a form as follows.
\begin{eqnarray}\label{exp1}\nonumber
\frac{\partial u}{\partial t} 
 =& \gamma q_k\frac{(u_x^2+u_y^2+(q_k-1)u_z^2)u_{zz}+(u_x^2+(q_k-1)u_y^2+u_z^2)u_{yy}
		+((q_k-1)u_x^2+u_y^2+u_z^2)u_{xx}}{u_x^2+u_y^2+u_z^2}\\ 
\nonumber		  &- \gamma(2-q_k)q_k\frac{2u_xu_yu_{xy}+2u_xu_zu_{xz}+  2u_zu_yu_{yz}}
	{u_x^2+u_y^2+u_z^2}\\
	&+\left(\sqrt{u_x^2+u_y^2+u_z^2}\right)^{2-q_k} p u^{p-1} V.
\end{eqnarray}
In particular, the time-dependent derivative is carried out by explicit Euler scheme. Note that other implicit schemes can be designed to improve the solution efficiency and will be pursued later. The first and second order spatial derivatives are handled by finite difference schemes \cite{ZhanChen:2010a}. To impose the domain decomposition in  \eqref{EL-sys}, we let $u$  be fixed as one in the pure solute area $\Omega_m$ and as zero in the pure solvent region $\Omega_s$. Here the pure solute area is numerically defined to be enclosed by a smoothed Van Der Waals surface (vdW) and the the pure solvent region is the area outside a smoothed solvent accessible surface (SAS). The initial value of $u$ in between  $\Omega_m$ and $\Omega_s$ can be set between 0 and 1. 

%In practice, a few smoothing steps are applied to make $\Sigma_0 $ and $ \Sigma_1$  {to be $C^2$}. 

\subsubsection{Coupling of the perturbed Poisson Boltzmann and surface evolution equations}\label{section 6.3.3}

In principle, the surface evolution equation needs to be
solved simultaneously with the perturbed PB equation until the solution process reaches a self-consistency. To speed up the whole iterative process, electrostatic potential $\psi $ is updated after a number of time steps (i.e., 10 to 100 steps) evolution of  the parabolic surface equation \cite{ZhanChen:2010a}.

Moreover, a simple relaxation algorithm is adopted to guarantee the convergence of the iterative process as follows \cite{ZhanChen:2010a} : 
\begin{align}
\label{relax}
	u=\alpha u_{new}+(1-\alpha)u_{old}, & \quad 0<\alpha<1,\\
	\label{eqn:relax2}
	\psi=\alpha' \psi_{\rm new}+(1-\alpha')\psi_{\rm old}, &\quad 0<\alpha'<1,
\end{align}
where $u_{new}$ and $u_{old}$ are the new and old $u$ profile values from current and previous steps, respectively. $\psi_{\rm old}$ and $\psi_{\rm new}$ denote previous and new  electrostatic potentials, respectively.  $\alpha=0.5$ and $\alpha'=0.5$  are set in our calculation. 

In addition, a simple cutoff strategy is conducted to apply   Constraint (\ref{constrain 1}) and to  avoid possible numerical errors:  
\begin{equation}\label{cutoff}
	u = 
	\begin{cases}
		u(x) \quad & u\in [0,1]\\
		0     & u<0\\
		1     & u>1.
	\end{cases}   
\end{equation} 
The cutoff checkup is carried out every time step or several steps during the solution of surface evolution equation.
 
Finally, to reduces the total iteration number and
save the computational time significantly, first of all, one may start the iterative process with an initial $u$ from solving Eq. (\ref{see}) without the electrostatic potential term. 
Second, one may take the prior potential $\psi$ as a good guess for the next resulted linear system in the PPB solution. That will make the PPB solver converge faster.

%a very small number, such as $10^{-6}$, is added to the denominator
%to avoid possible zeros in the denominator \cite{ZhanChen:2010a}. 

\subsubsection{Parameterization}\label{section 6.3.4 parameterization}

There are some parameter values that need to be determined for real numerical simulations of solvation free energy. They include solvent density $\rho_s$, the solvent radius  $\sigma_s$,$\gamma$,  {$P_h$} and so on. Since most of the parameters are involved in nonpolar solvation energy,  a previous simple parameter fitting strategy is adopted here \cite{ZhanChen:2016a, Shao2022constrained}. In particular, on the one side, some  parameter values are fixed or given such as: $\rho_s$=0.03341/\AA$^3$; solvent radius $\sigma_s$ =0.65 \AA~; radii of solute atoms like $\sigma_c$ =1.87 \AA~. On the other side, some are considered as fitting parameters like $\gamma$,  {$P_h$}, and well depth parameters $\epsilon_{is}$ where $i$ denotes different atom types. The following iterative procedure is used to obtain the optimal fitting parameter values: 

Step 0: An initial guess of fitting parameters and a trial set of molecules are determined with their existing informaion such as atomic coordinates, radii, and experimental data of solvation free energies. 

Step 1: For individual $j$-th molecule, $j=1,\cdots N_m$ where $N_m$ is the total number of molecules in the trial set, the coupled system \eqref{EL-sys} is solved until self-consistency is reached to find the quasi-steady state solution of $u_j$ and $\psi_j$ with latest parameter values. Note that if the trial set is nonpolar, one only needs to solve the surface evolution equation without a driven potential from the electrostatic field. Then the fitting process is exactly the same as our previous paper \cite{Shao2022constrained}. 

Step 2: Electrostatic solvation energy $I_{p,q_k}^{j} $ is calculated for each molecules using the profile of $\psi_j$.

Step 3: A non-negative least squares algorithm is used to update all non-negative  parameters  {$P_h $}, $\gamma$, and $\epsilon_{is}$ with a minimization problem
$$
	T = \underset{(p,\gamma,\epsilon_{is})}{\min}\sum\limits_{j=1}^{N_m}\left( I_{np,q_k}^j+I_{p,q_k}^j-I_{q_k}^{j,\rm exp}
	\right)^2,
	$$
where $I_{q_k}^{j,\rm exp}$ is the existing experimental data of solvation free energies in the literature.  
%{\color{red}(Why not $I_{q_k}^{j,\rm exp}$?)}

Step 4: The iterative loop from Step 1 to Step 3 is repeated until all fitting parameters converge to a certain set of values within a pre-set tolerance. 
	
	%Step 1: For the $j$--th molecule, we solve  \eqref{gfe} to find the solution of $u_j$ and then calculate solvation free  energy {$I_{np,q}^j$} with current parameter values based on \eqref{totale}.\\
	%(c) Set up a target function
	%{$$
	%T = \underset{(p,\gamma,\epsilon_{is})}{\min}\sum_j\left( I_{np,q}^j-I_{np,q}^{j,\rm exp}
	%\right)^2,
	%$$}
	%where {$I_{np,q}^{j,\rm exp}$ is the} experimental data of solvation free energies for the $j$--th molecule.\\
	%(d) All parameters  {$P_h $}, $\gamma$, and $\epsilon_{is}$ are updated  by resolving a nonnegative least squares
	%problem to determine those non-negative parameters. \\
	%(e) The iterative procedure (b)-(d) continues until convergence reaches within a pre-set tolerance for the above fitting parameters.	
	
\subsection{Simulation Results}

	In this section, both nonpolar and polar molecules are taken for the numerical simulation and model validation. Nonpolar molecules are simulated first to justify the usage of $u^p$ which represents the volume ratio of solute. That may minimize modeling uncertainties from solvent-solute electrostatic interactions. It is followed by the calculation of polar molecules to demonstrate the potential of current proposed model for the prediction of polar solvation energies. 
	
\subsubsection{Nonpolar molecules}	

 % Alkane 11   (N=40, p=1.00001)
	To validate the current constrained variational model, we start with a set of 11 alkanes as a calibration set for numerical implementation of model solution and the associated parameterizaton process. First of all, two parameters  $N$ and $q_k $ need  to be pre-determined for each simulation. It turns out that optimal fitting parameters are uniquely computed for  a set of arbitrary $N> 1$ and $q_k$, where $p=\frac{2N}{2N-1}$, and $q_k \to 1^+$. For instance, when $N=40$ and $q_k=1.00001$, the calculated optimal fitting parameters are the following: 	$\gamma=0.0746 $ kcal/(mol \AA$^2$), {$P_h=0.0090$} kcal/(mol \AA$^3$) and $ \epsilon_{cs}=0.486$ kcal/mol, and $\epsilon_{hs}=0.00$ kcal/mol. 
Note that $\epsilon_{hs}$ and $\epsilon_{cs}$ are well depth parameters of the hydrogen and carbon, respectively.
Moreover, it is shown that the current model is able to reproduce the total solvation free energies of 11 alkanes very well (see Table \ref{tablealkane11}). The root mean square (RMS) error of 11 alkenes is  0.109 kcal/mol. For the nonpolar solvation free energy, the repulsive and attractive parts of solvation free energy are also calculated for detailed comparisons with others in the literature. Note that the first two terms of \eqref{nonpolar energy} are considered as the repulsive part of solvation free energy. 
		
	\begin{table}[hbt!]
		\caption{Computed total solvation free energies of the trial set of 11 alkane compounds and their repulsive and attractive decomposition when $q_k=1.00001$. 
			$\gamma=0.0746 $ kcal/(mol \AA$^2$), {$P_h=0.0090$} kcal/(mol \AA$^3$) and $ \epsilon_{cs}=0.486$ kcal/mol, and 
			$\epsilon_{hs}=0.00$ kcal/mol
		}
		\begin{center}
			\begin{tabular}{|c|cccc|}
				\hline\hline Compound & Rep. part & Att. part & Numerical &
				Experimental \cite{Cabani:1981}
				\\\cline{2-5}
				&\multicolumn{4}{c|}{(kcal/mol)}\\
				\hline
				methane	&4.21	&-2.21	&2.00	& 2.00\\
				ethane	&5.90&-3.95	&1.95	&	1.83\\
				propane	&9.00	&-6.89	&2.12	&	1.96\\
				butane	&7.45 &-5.42	&2.03	&	2.08\\
				pentane	&10.58	&-8.27	&2.30	&	2.33\\
				hexane	&12.13	&-9.75	&2.38	&	2.49\\
				isobutane	&8.90	&-6.64	&2.26	&	2.52\\
				2-methylbutane	&10.20	&-7.80	&2.40	&	2.38\\
				neopentane	&10.21 &-7.61	&2.60	&	2.50\\
				cyclopentane	&9.21	&-8.04	&1.17	&	1.20\\
				cyclohexane	&10.45	&-9.08	&1.37	&	1.23\\
				RMS of calibration set      &&&0.109&\\
				\hline\hline
			\end{tabular}
		\end{center}
		\label{tablealkane11}
	\end{table}			
	
    Next, it is interesting to see whether the model parameter $N$ or equivalently $p=\frac{2N}{2N-1}$, which is introduced in the volume ratio of solute $u^p$,  plays an important role in the solvation free energy calculation and prediction.  For this purpose, different $N$ values are chosen for the set of 11 alkanes while fixing all other simulation setting. It is evident that almost identical simulation results are obtained for large enough $N$  (See Table \ref{various_q}). %explain some reasoning? 
   	\begin{table}[hbt!]
		\caption{Different optimized parameters and RMS errors for various $N$ values
	 when $q_k=1.00001$	}
		\begin{center}
			\begin{tabular}{|c|cccc|}
				\hline\hline { $q$} value & $\gamma$ (kcal/(mol \AA$^2$)) & $P_h$ (kcal/(mol \AA$^3$)) & $ \epsilon_{cs}$ (kcal/mol) &
				RMS (kcal/mol)\\ \hline
				1	&0.0758	&0.0078	&0.493	& 0.105\\
				2	&0.0749	&0.0085 & 0.487	&	0.108\\
				5	&0.0746 &0.009	&0.486 &	0.109 \\
				10  &0.0746 &0.009	&0.486 &	0.109 \\
				20 &0.0746 &0.009	&0.486 &	0.109 \\
				40	&0.0746 &0.009	&0.486 &	0.109 \\
				\hline\hline
			\end{tabular}
		\end{center}
		\label{various_q}
	\end{table}					
		
	Moreover,  with $q=1.00001$ and $N=40$, a predictive study is conducted for a set of 11 alkene compounds which was also used before \cite{Ratkova:2010, ZhanChen:2016a, Shao2022constrained}. The assumed similar solvent environment allows one to apply {the} above-obtained optimized parameters of 11 alkanes here because of the fact that both nonpolar sets only possess two types of atoms (C and H). It turns out that the numerical prediction of the current model matches the experimental data well as shown in Table \ref{tablealkene11}. The RMS error of 11 alkenes is 0.21 kcal/mol.

	\begin{table}[hbt!]
		\caption{Computed total solvation free energies of 11 alkene compounds { when $q=1.00001$} and $N=40$.
		}
		\begin{center}
			\begin{tabular}{|c|cccc|}
				\hline\hline Compound & Rep. part & Att. part & Numerical &
				Experimental \cite{Ratkova:2010}
				\\\cline{2-5}
				&\multicolumn{4}{c|}{(kcal/mol)}\\
				\hline
				3-methyl-1- butene	& 10.15	&-8.32	&1.84	&1.82\\	
				1-butene	&8.68 	&-7.04	&1.64	&            1.38\\
				ethene	  &5.51	&-4.12	  &1.49 &              1.27\\	
				1-heptene	&13.42	&-11.58	&1.84	&          1.66\\	
				1-hexene	&11.83	&-10.05 &1.78	&          1.68\\
				1-nonene	&16.64	&-14.59	&1.95	&          2.06\\	
				2-methyl-2-butene	&10.08	&-8.33	&1.74	&  1.31\\	
				1-octene	&14.99	&-13.01	&1.98	&          2.17\\	
				1-pentene	&10.22	&-8.58	&1.65	&          1.66 \\	
				1-propene	&7.12   &-5.59	&1.53	&            1.27\\	
				trans-2-heptene	&13.45	&-11.62	&1.83	&    1.66\\	
				RMS of prediction set   & & & 0.209& \\
				
				\hline\hline
			\end{tabular}
		\end{center}
		\label{tablealkene11}
	\end{table}	
	
%	\begin{figure}
%		\begin{center}
%			\begin{tabular}{cc}
%%				\includegraphics[width=1.0\columnwidth]{png/alkane_alkene_corr.png}
%				%\includegraphics[width=0.6\columnwidth]{alkane_alkene_corr.png}
%			\end{tabular}
%		\end{center}
%		\caption{
%			Comparison of computed and experimental data of solvation free energies of 11 alkanes and 11 alkenes when q=1.
%		}\label{figalkenesalkanes}
%	\end{figure}

	Furthermore, we have theoretically proved that total solvation energies converge to the case of $q_k=1$ when $q_k\to 1^+$. Numerically, the convergence can be demonstrated as follow: choosing a set of molecules like the above alkene compounds and fixing all other numerical settings, one allows the value of $q_k$ to approach 1 by creating a sequence of $q_k$ ($q_k=1.01, 1.001, 1.0001, 1.00001, 1.000001$). Then the total solvation free energy of each molecule is computed.  Table \ref{convergencetable} illustrates the convergence of total solvation free energies for all eleven alkenes.  
		
	\begin{table}[hbt!]
		\caption{Convergence of total solvation free energies {of} of eleven alkene molecules when  $q \to 1^+$ with other  parameter values fixed. 
		}
		\begin{center}
			\begin{tabular}{|c|c|c|c|c|c|}
				\hline\hline Compound & 1.01 &  1.001& 1.0001 &1.00001 & 1.000001
				\\\cline{2-6}
				&\multicolumn{5}{c|}{(kcal/mol)}\\
				\hline
				3-methyl-1- butene	 &2.567 	&1.908   &1.844 & 1.837 &1.837 \\	
				1-butene	&2.268 	&1.701	&1.647	& 1.641  &1.641\\
				ethene	  &1.888	&1.524	  &1.489 &  1.485&1.485 \\	
				1-heptene	&2.797	&1.930	&1.846	&1.837   & 1.837\\	
				1-hexene	&2.625  &1.857 &1.784	& 1.776  & 1.775\\
				1-nonene	&3.126	&2.060	&1.957	& 1.946  & 1.946 \\	
				2-methyl-2-butene	&2.468	&1.751	&1.744	&  1.745 &1.745 \\	
				1-octene	&3.049	&2.083	&1.990	& 1.980 &  1.980\\	
				1-pentene	&2.381	&1.716	&1.653	&  1.646  & 1.645 \\	
				1-propene	&2.043   &1.575	&1.530	&   1.525   & 1.525\\	
				trans-2-heptene	&2.789	&1.918	&1.835	& 1.826 & 1.826\\	
				\hline\hline
			\end{tabular}
		\end{center}
		\label{convergencetable}
	\end{table}	
	
	Remark that regarding the numerical calculation of solvation free  energy for nonpolar molecules, the currently computed results are almost the same as the previous constrained solvation model \cite{Shao2022constrained} when $N$ is large enough. The similarity can be explained by the fact that $pu^{p-1} \to 1 $ for $0<u<1$ when $p=\frac{2N}{2N-1} \to 1 $ with $ N \to \infty$.
	 
\subsubsection{Polar molecules}
The introduction of $u^p$ as solute volume ratio enables us to derive the system~\eqref{EL-sys}  from proposed constrained total  solvation   energy model~\eqref{non penalized total energy functional final form}. 
It has been a theoretical advance from our previous constrained model in which a PDE  was derived only for nonpolar energy functional due to the complex two-obstacle problem \cite{Shao2022constrained}. 

In this section, the model potential and validation are demonstrated numerically for polar molecules. To the end, a challenging set of 17 compounds is chosen. The challenge arises partially due to strong solvent-solute interactions caused by polyfunctional or interacting polar groups. Actually, its challenge can be seen quantitatively. For instance, using an explicit solvent model, Nicholls et al. obtained the root mean square error (RMS) as  $1.71 \pm 0.05 $ kcal/mol  via \cite{Nicholls:2008}. With an improved multiscale model equipped with self-consistent quantum charge density by Chen et al \cite{ZhanChen:2011b}, RMS was still around 1.50 kcal/mol.

For the current simulation, the structure data of the set of 17 molecules is taken from the supporting information of the paper of Nicholls et al \cite{Nicholls:2008} as we did before. The dielectric constants are slightly adjusted.
In the solute region  $\epsilon_m\approx 1$, while $\epsilon_s \leq 80$ for the solvent region.  
For this 17 set, different well-depth parameters $\epsilon_{is}$ need to be optimized  based on the above-described simple parameterization scheme.  
It is shown that the computed solvation free energy is quite comparable with the experimental data. 
The root mean square error can be improved to 1.107 kcal/mol (See table \ref{table17set}) when $\epsilon_m=1.15$ and $\epsilon_s=70$. 
In addition, it is found that almost identical simulation results are obtained for large enough $N$. 
In other words, model parameter value $N$ does not play an important role for the solvation energy prediction while it obviously benefits the theoretical derivation and the proof for current constrained variational model.  
The minor effect of different $N$ values can be found in Table \ref{nconvergence_polar}.

\begin{table}[hbt!]
	\caption{Comparison of total free energies (kcal/mol) for 17 compounds}\label{table17set}
	\begin{center}
		\begin{tabular}{cccc}
			\hline \hline Compound &               $\Delta G$    & Exptl  & Error  \\
			\hline
			glycerol triacetate          &-10.10       & -8.84  & -1.26 \\
			benzyl bromide               &-2.38        & -2.38  & 0.00   \\
			benzyl chloride              &-3.95        & -1.93  & -2.02 \\
			m-bis(trifluoromethyl)benzene         &1.07         &  1.07  & 0.00  \\
			N,N-dimethyl-p-methoxybenzamide     &-8.74        & -11.01 &  2.27    \\
			N,N-4-trimethylbenzamide  &-8.60        & -9.76  &  1.16   \\
			bis-2-chloroethyl ether   &-3.26        & -4.23  &  0.97    \\
			1,1-diacetoxyethane       &-5.49        & -4.97   & -0.52  \\
			1,1-diethoxyethane        &-4.51       & -3.28  &  -1.23 \\
			1,4-dioxane               &-4.84        & -5.05  &  0.21 \\
			diethyl propanedioate    &-5.10        & -6.00  & -0.90    \\
			dimethoxymethane          &-1.28        & -2.93   & 1.65  \\
			ethylene glycol diacetate &-6.48       & -6.34  &  -0.14  \\
			1,2-diethoxyethane        &-4.64       & -3.54  &  -1.10   \\
			diethyl sulfide           &-1.43        & -1.43  &  0.00 \\
			phenyl formate            &-4.35       & -4.08  & -0.27 \\
			imidazole                 &-10.83      & -9.81  & -1.02\\
				RMS of 17 polar molecules & & 1.107& \\
			\\ \hline\hline
		\end{tabular}
	\end{center}
\end{table}

	\begin{table}[hbt!]
	\caption{Some optimized parameters and RMS errors from various $N$ values
		when $q_k=1.00001$	}
	\begin{center}
		\begin{tabular}{|c|cccc|}
			\hline\hline { $q$} value & $\gamma$ (kcal/(mol \AA$^2$)) & $P_h$ (kcal/(mol \AA$^3$)) & $ \epsilon_{cs}$ (kcal/mol) &
			RMS (kcal/mol)\\ \hline
			4	&0.314 &0.000	&1.105 &	1.107\\
			8 	&0.314 &0.000	&1.105 &	1.107\\
			16 	&0.314 &0.000	&1.105 &	1.107\\
			32		&0.314 &0.000	&1.105 & 1.107\\
			\hline\hline
		\end{tabular}
	\end{center}
	\label{nconvergence_polar}
\end{table}					

%17set parameterized 

% N convergence of 17 set in terms of parameter values and RMS
%%%%%%%%%%%%%%%%%%%%%%%%%%%%%%%%%%%%%%%%%%%%%%%%%%%%%%%%%%%%
\section{Conclusions}\label{Section 7}
Variational implicit solvation models (VISM) with diffuse solvent-solute interface definition have been considered as a successful approach to compute the disposition of an interface separating the solute and the solvent. It has been shown numerically that variational diffuse-interface solvation models can significantly improve the accuracy and efficiency of solvation energy computation. However, there are several open questions concerning those models at a theoretic level. In particular, 
all existing VISMs in literature lack the uniqueness of an energy minimizing solute-solvent interface and thus prevent us from studying many important properties of   the  interface profile.

Therefore, by introducing a new volume ratio function $u^p$, in this work, we have developed a novel constrained VISM based on a promising previously-proposed total variation based model (TVBVISM). Existence, uniqueness and regularity of the energy minimizing solute-solvent interface have been studied. 
%We have proved that the new model is strictly convex and thus admits at most one minimizer via introducing a new parameter $p$. 
Moreover, with the assistance of the precise depiction of   the  interface profile, this work provides a partial answer to the question  why the   solvation free energy is not minimized by a sharp solute-solvent interface. It turns out that when the mean curvature of $\Sigma_0$ is positive at some point, the energy minimizing state is never achieved by a sharp interface. 

In addition,  for the variational analysis of the new model and for the numerical computation of the solvation energy, a novel  approach has been proposed to overcome the essential difficulty generated by the involved constraints in the model. Specifically, the variational formulas of the new   energy functional can be rigorously derived via the introduction of the new volume ratio function $u^p$ together with an approximation technique by a sequence of $q$-energy type functionals. This is another advance from our previous work in which only the numerical study of nonpolar energy can be conducted for a constrained VISM. Model validation and numerical implementation have been demonstrated by using several common biomolecular modeling tasks. Numerical simulations show that the solvation energies calculated from our new model match the experimental data very well.

For the future work, we will provide a complete proof for the continuous dependence of   the solvation free energy      on the surfaces $\Omega_m$ and $\Omega_s$ in a suitable topology. 
Numerically, based on the derived   elliptic  system, we  intend to further improve the accuracy and efficiency of the solvation energy prediction via refined parameterization schemes.
Moreover, analysis of the current and potential numerical schemes like convergence will be a topic for future study.

%%%%%%%%%%%%%%%%%%%%%%%%%%%%%%%%%%%%%%%%%%%%%%%%%%%%%%%%

\appendix

%%%%%%%%%%%%%%%%%%%%%%%%%%%%%%%%%%%%%%%%%%%%%%%%%%%%%%%%
\section{BV-functions}\label{Appendix A}

In Appendix~\ref{Appendix A}, we will introduce some notations and preliminaries of $BV-$functions. The main reference is \cite{MR1158660, MR1857292}.
Let $\Omega\subset \R^N$ be open. 

\begin{definition}
The space of  functions of bounded variations on $\Omega$, denoted by $BV(\Omega)$, is the collections of all $L^1(\Omega)-$functions whose gradient $D f$ in the sense of distributions is a
(vector-valued) Radon measure with finite total variation in $\Omega$.
The total variation of $f$ in $\Omega$ is defined by
$$
\sup\left\{ \int_\Omega  f\div z \, dx: \, z\in C_0^\infty(\Omega;\R^N), \, \|z\|_\infty \leq 1 \right\}
$$
and is denoted by $\|Df\|(\Omega)$ or $\int_\Omega\, d|Df|$. 
$BV(\Omega)$ is a Banach space endowed with the norm
$$
\|f\|_{BV}:=\|f\|_1 + \|D f \| (\Omega).
$$
\end{definition}

By the structure theorem of $BV-$functions, for every $f\in BV(\Omega)$, there exist  Radon measure $\mu$ and  a $\mu-$measurable function $\sigma:\Omega\to \R^N$ such that
\begin{itemize}
\item $|\sigma(x)|=1$ a.e. and
\item $\int_\Omega f \div z \, dx=-\int_\Omega z\cdot \sigma \, d\mu$ for all $z\in C_0^\infty(\Omega;\R^N)$.  
\end{itemize}
We write $|Df|$ for the measure $\mu$.

Sobolev embedding  also holds for functions of bounded variations:
\begin{equation}
\label{Sobolev embedding}
BV(\Omega) \hookrightarrow L^p(\Omega) ,\quad \text{for all } 1\leq p \leq 1^*=\frac{N}{N-1}.
\end{equation}
The embedding is compact when $1\leq p <1^*$.

\medskip
	
\begin{prop}\label{A1: compact of BV}
Let $\Omega$ be bounded and with Lipschitz boundary. Assume that $\{f_n\}_{n=1}^\infty \subset BV(\Omega)$ satisfies
$$
\sup\limits_n\|f_n\|_{BV} <\infty.
$$
Then there exists a subsequence, not relabelled, such that
$$
f_n\to f \quad \text{in } L^1(\Omega) \quad \text{for some }f\in BV(\Omega).
$$
\end{prop}

\medskip

\begin{prop}\label{A1: lsc of BV}
Suppose that $\{f_n\}_{n=1}^\infty \subset BV(\Omega)$ and $f_n\to f$ in $L^1_{loc}(\Omega)$. Then
$$
\|Df\|(\Omega)\leq \liminf\limits_{n\to \infty}\|Df_n\|(\Omega).
$$
\end{prop}

An Lebesgue measurable set $E \subset \R^N$ is said to have {\em finite perimeter} in $\Omega$ if
$$
\chi_E \in BV(\Omega).
$$
$\Per(E;\Omega):=\|D\chi_E \|(\Omega)$ is called the {\em perimeter} of $E$ in $\Omega$.
%Given $f\in BV(\Omega)$, define $E_s=\{f\geq s\}$. Then $E_s$ has finite perimeter for a.a. $s\in \R$

\medskip

\begin{definition}
Let $E$ be of finite perimeter in $\Omega$. We call the {\em reduced boundary} $\partial E^*$ the collection of all points $x\in {\rm supp}|D \chi_E|\cap \Omega$ such that the limit
$$
\nu_E(x):=-\lim\limits_{r\to 0^+ } \frac{D \chi_E (B(x,r)) }{ \|D \chi_E\| (B(x,r)) }
$$
exists in $\R^N$ and satisfies $|\nu_E|(x)=1$ a.e.. 
The function $\nu_E: \partial E^* \to \mathbb{S}^{N-1}$ is called the generalized {\em outer normal} to $E$.
$\partial E \setminus \partial E^*$ is called the {\em singular set} of $E$.
In particular, we have  
\begin{equation}\label{red bdry}
\Per(E;\R^N\setminus \partial E^*)=0.
\end{equation}
\end{definition}

\medskip

\begin{prop}\label{A1: trace of BV}
Let $\Omega$ be bounded and with Lipschitz boundary.  
There is a bounded linear map
$$
\Tr: BV(\Omega)\to L^1(\partial \Omega)
$$
such that
$$
\int_\Omega f \div \phi\, dx=- \int_\Omega \phi\cdot Df + \int_{\partial\Omega}  (\phi\cdot \nu) \Tr f\, d\cH^{N-1},
$$
where $\nu$ is the outer unit normal on $\partial\Omega$. 
It is understood that the measure on $\partial\Omega$ is $\cH^{N-1}$.
The function $\Tr f$, which is uniquely defined  $\cH^{N-1}$ a.e. on $\partial \Omega$, is called the {\em trace} of $f$ on $\partial \Omega$.
\end{prop}

\medskip

\begin{prop}\label{A1: extension of BV}
Let $\Omega$ be bounded and   Lipschitz.  Assume that $f_1\in BV(\Omega)$ and $f_2\in BV(\R^N\setminus \overline{\Omega})$. Define
\begin{align*}
f(x)=
\begin{cases}
f_1(x) \quad & x\in \Omega\\
f_2(x) & x\in \R^N\setminus \overline{\Omega}.
\end{cases}
\end{align*}
Then $f\in BV(\R^N)$. Moreover,
$$
\|Df\|(\R^N)=\|Df_1\|(\Omega) + \|Df_2\|(\R^N\setminus \overline{\Omega}) + \int_{\partial\Omega} |\Tr f_1-\Tr f_2|\, d\cH^{N-1}.
$$
\end{prop}

Given $f\in L^1_{loc}(\Omega)$, we say that $f$ has an {\em approximate limit} at $x\in \Omega$ if there exists $z\in \R$ such that
\begin{equation}
\label{A1: approx limit}
\lim\limits_{r\to 0^+} \frac{1}{|B(x,r)|} \int_{B(x,r)} | u(y)-z|\, dy=0.
\end{equation}
The set of points where this does not hold is called the {\em approximate discontinuity}
set of $f$, and it is denoted by $S_f$. By Lebesgue differentiation theorem, $\cL^N(S_f)=0$.
$z$ is uniquely determined via \eqref{A1: approx limit} and is denoted by $\tilde{f}(x)$.
$f$ is said to be {\em approximately continuous} at $x$ if $x\notin S_f$ and $f(x)=\tilde{f}(x)$.

We say $f\in L^1_{loc}(\Omega)$ has an {\em approximate jump point} at $x\in \Omega$ if there exist $a\neq b\in \R$ and $\mu\in \mathbb{S}^{N-1}$ such that $a\neq b$ and
$$
\lim\limits_{r\to 0^+} \frac{1}{|B(x,r)|} \int_{B^+_\nu(x,r)} | f(y)-a|\, dy=0 
\quad \text{and}\quad 
\lim\limits_{r\to 0^+} \frac{1}{|B(x,r)|} \int_{B^{-}_\nu(x,r)} | f(y)-b|\, dy=0.
$$
Here
\begin{align*}
\begin{cases}
B^+_\nu(x,r):=& \{y\in B(x,r): \,  \nu\cdot (y-x)>0\}\\
B^{-}_\nu(x,r):=& \{y\in B(x,r): \, \nu \cdot(y-x)<0\} .
\end{cases}
\end{align*}
%with $ \cdot$ being the inner product in $\R^N$.
The set of all approximate jump points of $f$ is denoted by $J_f$.
When $f\in BV(\Omega)$, $S_f$ is countably $\cH^{N-1}-$rectifiable and $J_f$ is a Borel subset of $S_f$. 
Further $\cH^{N-1}(S_f\setminus J_f)=0$.

If $f\in BV(\Omega)$, we define the super-level sets of $f$ by
$$
E_t:=\{f > t\}, \quad t\in \R.
$$
Then for $\cL^1-$a.a. $t$, $E_t$ is of finite perimeter and the function 
$$
[t\mapsto \Per(E_t;\Omega)]
$$ 
is $\cL^1-$measurable. Moreover, the {\em coarea formula} holds:
\begin{equation}
\label{coarea formula}
\int_\Omega v(x)d|Du| =\int_{-\infty}^\infty \int_\Omega v(x)d|D\chi_{E_t}| \, dt 
\end{equation} 
for all $|Du|-$integrable function $v:\Omega\to \R$.
In addition,
\begin{equation}
\label{jump set}
J_f=\bigcup\limits_{t_1,t_2\in \mathbb{Q} , \, t_1<t_2 } \partial E_{t_1} \cap \partial E_{t_2}.
\end{equation}

If $E\subset \R^N$ is measurable, we can define the upper and lower density of $E$ at $x$ by
$$
\overline{D}(E,x)= \limsup\limits_{r\to 0^+} \frac{|E\cap B(x,r)|}{|B(x,r)|} 
\quad \text{and}\quad 
\underline{D}(E,x)= \liminf\limits_{r\to 0^+} \frac{|E\cap B(x,r)|}{|B(x,r)|} ,
$$
respectively. If $u\in BV(\Omega)$, we define 
$$
u^*(x)=\inf\{s:\, \overline{D}(\{u\geq s\},x)=0\} \quad \text{and}\quad 
  u_*(x)=\sup\{s:\, \underline{D}(\{u\leq s\},x)=0\}.
$$
Then $u$ is approximately continuous at $x\in \Omega$ iff $u^*(x)=u_*(x)$.

%\begin{prop}\label{A1: coarea}
%Let $f\in BV(\Omega)$ and $E_s=\{f>s\}$. Then
%\begin{enumerate}
%\item $E_s$ has finite perimeter for a.a. $s\in \R$ and
%\item $\|Df\|(\Omega)=\int_{-\infty}^\infty \Per(E_s;\Omega) \, ds$.
%\end{enumerate}
%\end{prop}

%%%%%%%%%%%%%%%%%%%%%%%%%%%%%%%%%%%%%%%%%%%%%%%%%%%%%%%%
\section{Tools from convex analysis}\label{Appendix B}

In Appendix~\ref{Appendix B}, we will state some useful tools from Convex Analysis. Interested readers may refer to the books \cite{MR1727362, MR1422252} for more details.

Let $X$ be a Banach space with norm $\|\cdot\|$. 
Throughout, we assume that $f: X\to \R\cup \{\pm \infty\}$ is convex and lower semicontinuous (l.s.c.) function. 
Its {\em effective domain} is defined by
is
$$
\dom(f)=\{u\in X:\, f(u)<+\infty \}.
$$
$f$ is said to be {\em proper} if it nowhere takes value $-\infty$ and is not identically equal to $+\infty$ on $X$.

%The {\em epigraph} of $f$ is defined by
%\begin{equation}\label{def: epigraph}
%\epi(f)=\{(u,c)\in X\times \R: \, f(u)\leq c\}
%\end{equation}
%We endow the set $\epi(f)$ with the metric
%$$
%d((u,c). (v,r))= \sqrt{\|u-v\|^2 + (c-r)^2}.
%$$
%In addition, we define $\cG_f:\epi(f)\to \R$ by
%$$
%\cG_f(u,c)=c.
%$$
%$\cG_f$ is Lipschitz continuous of constant $1$.

Given any subset $U\subset X$, its {\em indicator function} $I_U$ is defined by
\begin{align}
\label{def: indicator function}
I_U(x)=
\begin{cases}
0 \quad & \text{when } x\in U\\
\infty & \text{when } x\in X\setminus U.
\end{cases}
\end{align}

We denote by $X^*$ the topological dual of $X$ and $\langle \cdot , \cdot \rangle$ the duality pairing. 
When $f$ is proper, the {\em subdifferential} of $f$ at $u\in \dom(f)$ is the set of all $u^*\in X^*$ such that
$$
\langle u^*, v-u \rangle \leq f(v)-f(u) ,\quad \forall v\in X,
$$
and is denoted by $\partial f(u)$. Each element of $\partial f(u)$ is called a {\em subdifferential} of $f$ at $u$. 
When $\partial f(u) \neq \emptyset$, we say that $f$ is {\em subdifferentiable} at $u$.

The relationship between subdifferentiability and G\^ateaux-differentiability is described by the following proposition.
\begin{prop}\label{Prop: Frechet}
Let $f:X\to \R\cup \{+\infty\}$ be convex and proper. If $f$ is G\^ateaux-differentiable at $u\in {\rm int}(\dom(f))$, then $\partial f(u)=f'(u)$, where $f'(u)$ is the G\^ateaux-derivative of $f$ at $u$.
\end{prop}

By the definition of the subdifferential, it is obvious that
$$
\partial f_1(v)+ \partial f_2(v) \subseteq \partial (f_1+f_2)(v).
$$
However, the converse is not always true. We list below several cases where the converse holds.
 
\begin{prop}\label{Prop: sum rule of subdifferential 1}
Suppose that $f_1,f_2: X \to \R\cup \{+\infty\}$ is convex and l.s.c. and $u\in \dom(F_1)\cap \dom(F_2)$.
If $f_2$ is continuous at $u$, then
$$
\partial f_1(v)+ \partial f_2(v) = \partial (f_1+f_2)(v) \quad \forall v\in X.
$$
\end{prop}
 
\begin{prop}\label{Prop: sum rule of subdifferential 2}
Let $f, g : X \to \R\cup\{\infty\}$ be proper, l.s.c. and convex functions such that 
$$
\bigcup\limits_{\mu>0} \mu(\dom(f)-\dom(g)) \text{ is a closed linear subspace of } X,
$$
then
$$
\partial(f+g)(u)=\partial f(u)+ \partial g(u) \quad \forall u\in \dom(f)\cap \dom(g).
$$
\end{prop}
\begin{proof}
This is \cite[Corollary 2.1]{MR849549}.
See also \cite{MR1823719} for an easy proof.
\end{proof}

%%%%%%%%%%%%%%%%%%%%%%%%%%%%%%%%%%%%%%%%%%%%%%%%%%%%%%%%
%\subsection{Tools from non-smooth analysis}

%%%%%%%%%%%%%%%%%%%%%%%%%%%%%%%%%%%%%%%%%%%%%%%%%%%%%%%% 
%\subsection{Tools from functional analysis}
%\begin{lem}\label{lem: A2}
%If $v_k\in \cA$ satisfies $\|v\|_{H^1}\to \infty$, then $\|\nabla v_k\|_2\to \infty$. 
%\end{lem}
%\begin{proof}
%Since $\|v_k-\psi_\infty\|_{H^1}\to \infty$, Poincar\'e inequality shows that 
%$$
%\|\nabla(v_k-\psi_\infty)\|_2\to \infty.
%$$
%Hence by Cauchy-Schwartz and Young's inequalities
%\begin{align*}
%\|\nabla v_k\|_2^2=& \|\nabla(v_k-\psi_\infty)\|_2^2 -2 \nabla(v_k-\psi_\infty)\cdot \nabla \psi_\infty + \|\psi_\infty\|_2^2 \\
%\geq & (1-\varepsilon) \|\nabla(v_k-\psi_\infty)\|_2^2  + (1-C(\varepsilon) )\|\psi_\infty\|_2^2 \to \infty
%\end{align*}
%as $k\to \infty$.
%\end{proof}
\section*{Acknowledgments}
This work is supported in part by National Science Foundation (NSF) grant No. DMS-1818748  (Z. Chen)

%\clearpage
\section*{Literature cited}
\renewcommand\refname{}

\bibliographystyle{plain}
\bibliography{refsshao}

%\section*{Acknowledge}

\end{document}